\documentclass[a4paper,twoside]{article}
\usepackage{a4}
\usepackage{amssymb}
\usepackage{amsmath}
\usepackage[active]{srcltx}
\usepackage[pagebackref,colorlinks,citecolor=blue,linkcolor=blue,urlcolor=blue]{hyperref}
\usepackage[dvipsnames]{color}
\usepackage{upref}
\allowdisplaybreaks[2] 
\newcount\minutes \newcount\hours
\hours=\time
\divide\hours 60
\minutes=\hours
\multiply\minutes -60
\advance\minutes \time
\newcommand{\klockan}{\the\hours:{\ifnum\minutes<10 0\fi}\the\minutes}
\newcommand{\tid}{\today\ \klockan}
\newcommand{\prtid}{\smash{\raise 10mm \hbox{\LaTeX ed \tid}}}
\renewcommand{\prtid}{}
\makeatletter
\def\sectionmark#1{} 
\def\subsectionmark#1{}
\newcommand{\sectnr}{\ifnum \c@secnumdepth >\z@
                 \thesection.\hskip 1em\relax \fi}
\def\@evenhead{\footnotesize\rm\thepage\hfil\leftmark\hfil\llap{\prtid}}
\def\@oddhead{\footnotesize\rm\rlap{\prtid}\hfil\rightmark\hfil\thepage}
\def\tableofcontents{\section*{Contents} 
 \@starttoc{toc}}
\makeatother
\makeatletter
\def\@biblabel#1{#1.}
\makeatother
\makeatletter
\let\Thebibliography=\thebibliography
\renewcommand{\thebibliography}[1]{\def\@mkboth##1##2{}\Thebibliography{#1}
\addcontentsline{toc}{section}{References}
\frenchspacing 
\setlength{\@topsep}{0pt}
\setlength{\itemsep}{0pt}%
\setlength{\parskip}{0pt plus 2pt}%
}
\makeatother
\makeatletter
\def\mdots@{\mathinner.\nonscript\!.%
 \ifx\next,.\else\ifx\next;.\else\ifx\next..\else
 \nonscript\!\mathinner.\fi\fi\fi}
\let\ldots\mdots@
\let\cdots\mdots@
\let\dotso\mdots@
\let\dotsb\mdots@
\let\dotsm\mdots@
\let\dotsc\mdots@
\def\vdots{\vbox{\baselineskip2.8\p@ \lineskiplimit\z@
    \kern6\p@\hbox{.}\hbox{.}\hbox{.}\kern3\p@}}
\def\ddots{\mathinner{\mkern1mu\raise8.6\p@\vbox{\kern7\p@\hbox{.}}%
    \raise5.8\p@\hbox{.}\raise3\p@\hbox{.}\mkern1mu}}
\makeatother
\makeatletter
\let\Enumerate=\enumerate
\renewcommand{\enumerate}{\Enumerate%
\setlength{\itemsep}{0pt}%
\setlength{\parskip}{0pt plus 1pt}%
\renewcommand{\theenumi}{\textup{(\alph{enumi})}}%
\renewcommand{\labelenumi}{\theenumi}%
}
\makeatother
\makeatletter
\def\@seccntformat#1{\csname the#1\endcsname.\quad}
\makeatother
\newcommand{\authortitle}[2]{\author{#1}\title{#2}\markboth{#1}{#2}}
\newcommand{\art}[6]{{\sc #1, \rm #2, \it #3\/ \bf #4 \rm (#5), \mbox{#6}.}}

\newcommand{\artin}[3]{{\sc #1, \rm #2,  in #3.}}

\newcommand{\arttoappear}[3]{{\sc #1, \rm #2, to appear in \it #3}}
\newcommand{\auth}[2]{{#1, #2.}}

\newcommand{\book}[3]{{\sc #1, \it #2, \rm #3.}}
\newcommand{\AND}{{\rm and }}

\RequirePackage{amsthm}
\newtheoremstyle{descriptive}%
  {\topsep}   
  {\topsep}   
  {\rmfamily} 
  {}          
  {\bfseries} 
  {.}         
  { }         
  {}          
\newtheoremstyle{propositional}%
  {\topsep}   
  {\topsep}   
  {\itshape}  
  {}          
  {\bfseries} 
  {.}         
  { }         
  {}          
\theoremstyle{propositional}
\newtheorem{thm}{Theorem}[section]

\newtheorem{lem}[thm]{Lemma}
\newtheorem{cor}[thm]{Corollary}
\theoremstyle{descriptive}
\newtheorem{deff}[thm]{Definition}
\newtheorem{remark}[thm]{Remark}

\makeatletter
\renewenvironment{proof}[1][\proofname]{\par
  \pushQED{\qed}%
  \normalfont
  \trivlist
  \item[\hskip\labelsep
        \itshape
    #1\@addpunct{.}]\ignorespaces
}{%
  \popQED\endtrivlist\@endpefalse
}
\makeatother
{\catcode`p =12 \catcode`t =12 \gdef\eeaa#1pt{#1}}      
\def\accentadjtext#1{\setbox0\hbox{$#1$}\kern   
                \expandafter\eeaa\the\fontdimen1\textfont1 \ht0 }
\def\accentadjscript#1{\setbox0\hbox{$#1$}\kern 
                \expandafter\eeaa\the\fontdimen1\scriptfont1 \ht0 }
\def\accentadjscriptscript#1{\setbox0\hbox{$#1$}\kern   
                \expandafter\eeaa\the\fontdimen1\scriptscriptfont1 \ht0 }
\def\accentadjtextback#1{\setbox0\hbox{$#1$}\kern       
                -\expandafter\eeaa\the\fontdimen1\textfont1 \ht0 }
\def\accentadjscriptback#1{\setbox0\hbox{$#1$}\kern     
                -\expandafter\eeaa\the\fontdimen1\scriptfont1 \ht0 }
\def\accentadjscriptscriptback#1{\setbox0\hbox{$#1$}\kern 
                -\expandafter\eeaa\the\fontdimen1\scriptscriptfont1 \ht0 }
\def\itoverline#1{{\mathsurround0pt\mathchoice
        {\rlap{$\accentadjtext{\displaystyle #1}
                \accentadjtext{\vrule height1.593pt}
                \overline{\phantom{\displaystyle #1}
                \accentadjtextback{\displaystyle #1}}$}{#1}}
        {\rlap{$\accentadjtext{\textstyle #1}
                \accentadjtext{\vrule height1.593pt}
                \overline{\phantom{\textstyle #1}
                \accentadjtextback{\textstyle #1}}$}{#1}}
        {\rlap{$\accentadjscript{\scriptstyle #1}
                \accentadjscript{\vrule height1.593pt}
                \overline{\phantom{\scriptstyle #1}
                \accentadjscriptback{\scriptstyle #1}}$}{#1}}
        {\rlap{$\accentadjscriptscript{\scriptscriptstyle #1}
                \accentadjscriptscript{\vrule height1.593pt}
                \overline{\phantom{\scriptscriptstyle #1}
                \accentadjscriptscriptback{\scriptscriptstyle #1}}$}{#1}}}}
\def\itunderline#1{{\mathsurround0pt\mathchoice
        {\rlap{$\underline{\phantom{\displaystyle #1}
                \accentadjtextback{\displaystyle #1}}$}{#1}}
        {\rlap{$\underline{\phantom{\textstyle #1}
                \accentadjtextback{\textstyle #1}}$}{#1}}
        {\rlap{$\underline{\phantom{\scriptstyle #1}
                \accentadjscriptback{\scriptstyle #1}}$}{#1}}
        {\rlap{$\underline{\phantom{\scriptscriptstyle #1}
                \accentadjscriptscriptback{\scriptscriptstyle #1}}$}{#1}}}}
\def\vint{\mathop{\mathchoice%
          {\setbox0\hbox{$\displaystyle\intop$}\kern 0.22\wd0%
           \vcenter{\hrule width 0.6\wd0}\kern -0.82\wd0}%
          {\setbox0\hbox{$\textstyle\intop$}\kern 0.2\wd0%
           \vcenter{\hrule width 0.6\wd0}\kern -0.8\wd0}%
          {\setbox0\hbox{$\scriptstyle\intop$}\kern 0.2\wd0%
           \vcenter{\hrule width 0.6\wd0}\kern -0.8\wd0}%
          {\setbox0\hbox{$\scriptscriptstyle\intop$}\kern 0.2\wd0%
           \vcenter{\hrule width 0.6\wd0}\kern -0.8\wd0}}%
          \mathopen{}\int}
\numberwithin{equation}{section}
\newenvironment{ack}{\medskip{\it Acknowledgement.}}{}
\renewcommand{\emptyset}{\varnothing}

\renewcommand{\phi}{\varphi}
\newcommand{\Ga}{\Gamma}
\newcommand{\eps}{\varepsilon}
\newcommand{\de}{\delta}
\newcommand{\ka}{\kappa}

\newcommand{\R}{\mathbf{R}}
\newcommand{\Z}{\mathbf{Z}}
\newcommand{\eR}{{\overline{\R}}}
\newcommand{\Rn}{\R^n}

\newcommand{\p}{{$p\mspace{1mu}$}}

\newcommand{\Np}{N^{1,p}}
\newcommand{\Nploc}{N^{1,p}\loc}
\newcommand{\Dp}{D^p}
\newcommand{\Dploc}{D^{p}\loc}
\newcommand{\Lploc}{L^{p}\loc}

\newcommand{\B}{{\cal B}}
\newcommand{\Bcalprime}{\B'}
\newcommand{\D}{{\cal D}}

\newcommand{\UU}{\mathcal{U}}%
\newcommand{\bdy}{\partial}
\newcommand{\bdry}{\partial}
\newcommand{\bdhat}{\widehat{\partial}}
\newcommand{\bdyhat}{\bdhat}
\newcommand{\setm}{\setminus}
\newcommand{\ga}{\gamma}
\newcommand{\la}{\lambda}
\newcommand{\La}{\Lambda}
\newcommand{\Om}{\Omega}
\newcommand{\dha}{\hat{d}}
\newcommand{\dhat}{\hat{d}}
\newcommand{\sh}{\hat{s}}

\newcommand{\muhq}{{\hat{\mu}_q}}
\newcommand{\muhtp}{{\hat{\mu}_{2p}}}
\newcommand{\mut}{{\tilde{\mu}}}
\newcommand{\gh}{\hat{g}}
\newcommand{\Xhat}{{\widehat{X}}}
\newcommand{\Xdot}{{\widehat{X}}}

\newcommand{\Bhat}{{\widehat{B}}}
\DeclareMathOperator{\Lip}{Lip}
\newcommand{\Lipc}{{\Lip_c}}
\newcommand{\loc}{_{\rm loc}}
\newcommand{\Cp}{{C_p}}
\newcommand{\CpBj}{{C_p^{B_j}}}
\newcommand{\Cpprime}{{C_p'}}
\newcommand{\Cphat}{{\widehat{C}_p}}
\newcommand{\Cphatprime}{{\widehat{C}_p'}}
\newcommand{\cp}{\capp_p}
\DeclareMathOperator{\capp}{cap}

\newcommand{\simge}{\gtrsim}
\newcommand{\simle}{\lesssim}

\newcommand{\Hp}{P}                 
\newcommand{\uP}{\itoverline{P}}     
\newcommand{\lP}{\itunderline{P}}

\DeclareMathOperator{\diam}{diam}
\DeclareMathOperator{\dvg}{div}
\newcommand{\grad}{\nabla}

\newcommand{\Bbar}{\itoverline{B}}
\newcommand{\chione}{\chi}
\newcommand{\binfty}{{\boldsymbol{\infty}}}
\newcommand{\CPI}{C_{\rm PI}}
\newcommand{\gat}{\tilde{\ga}}
\newcommand{\lahat}{\hat{\la}}

\newcommand{\uqi}{\overline{q}_\infty}
\newcommand{\Lhat}{\hat{L}}
\newcommand{\s}{\sigma}
\newcommand{\clOm}{\overline{\Om}}
\newcommand{\lmin}{l_{\min}}
\newcommand{\bCp}{{\protect\itoverline{C}_p}}
\newcommand{\A}{\mathcal{A}}
\newcommand{\clG}{\itoverline{G}}

\newcommand{\os}[2]{\overset{#1}{#2}}
\newcommand{\osr}[2]{\overset{\eqref{#1}}{#2}}
\begin{document}
%
%
\authortitle{Anders Bj\"orn, Jana Bj\"orn and Xining Li}
{Doubling measures and Poincar\'e inequalities 
for sphericalizations of metric spaces}
\title{Doubling measures and Poincar\'e inequalities 
for sphericalizations of metric spaces,
with applications to \p-harmonic functions in unbounded domains}
\author{
Anders Bj\"orn \\
\it\small Department of Mathematics, Link\"oping University, SE-581 83 Link\"oping, Sweden\\
\it \small anders.bjorn@liu.se, ORCID\/\textup{:} 0000-0002-9677-8321
\\
\\
Jana Bj\"orn \\
\it\small Department of Mathematics, Link\"oping University, SE-581 83 Link\"oping, Sweden\\
\it \small jana.bjorn@liu.se, ORCID\/\textup{:} 0000-0002-1238-6751
\\
\\
Xining Li \\
\it\small School of Mathematics (Zhuhai), Sun Yat-sen University, 
\it \small Zhuhai, 519082, P. R. China\/{\rm ;} \\
\it \small lixining3@mail.sysu.edu.cn, 
ORCID\/\textup{:} 0000-0001-6956-3926
}

\date{Preliminary version, \today}
\date{To appear in \emph{Ann. Fenn. Math.}}

\maketitle

\noindent{\small
{\bf Abstract.}
The identification between the complex plane and the Riemann sphere
preserves holomorphic and harmonic functions and is a classical tool.
In this paper we consider
a similar mapping from an unbounded metric space $X$ to a bounded
space and show how it preserves \p-harmonic functions and Poincar\'e inequalities.
When $X$ is  Ahlfors regular, this was shown in our earlier paper
(\emph{J.\ Math.\ Anal.\ Appl.} {\bf 474} (2019), 852--875).
Here we only require  the much
weaker (and more natural) doubling property
of the measure.
Furthermore, we consider a broader class of transformed measures.
The sphericalization is then applied  to 
obtain new results for the Dirichlet boundary value problem in unbounded sets
and for boundary regularity at infinity for \p-harmonic functions.
Some of these results are new also for 
unweighted $\mathbf{R}^n$, $n \ge 2$ and $p\ne2$.
}

\medskip

\noindent
{\small \emph{Key words and phrases}:
Boundary regularity,
doubling measure,
Dirichlet problem,
metric space,
Perron solution,
\p-harmonic function,
Poincar\'e inequality,
sphericalization.
}

\medskip

\noindent {\small \emph{Mathematics Subject Classification} (2020):
Primary: 
31E05; 
Secondary: 
26D10, 
30L10, 
30L15, 
31C45, 
35J66, 
35J92, 
49Q20. 
%
%
%
}



\section{Introduction}

The identification between the complex plane and the Riemann sphere
preserves
holomorphic and harmonic functions and is a classical tool with a vast
number of applications.
In this paper we consider
a similar mapping from an unbounded metric space to a bounded
space and show how it
preserves \p-harmonic functions, which in the metric space setting 
are defined as minimizers of the \p-energy integral and correspond to the solutions 
of the \p-Laplace equation 
\[
\dvg(|\grad u|^{p-2}\grad u) =0, \quad  1<p<\infty, 
\]
in Euclidean spaces.

This transformation opens up for future applications. 
We apply it to 
obtain new results for the Dirichlet boundary value
problem and boundary regularity
for \p-harmonic functions in unbounded open sets,
see Theorems~\ref{thm-resol-C-intro}--\ref{thm-barrier+local-intro}
and Section~\ref{sect-pharm}.
Some of these results are new also for 
unweighted $\Rn$, $n \ge 2$ and $p\ne2$.

Assume for the rest of the introduction that
$X=(X,d)$ is an \emph{unbounded} metric space equipped with a Borel measure $\mu$
such that $0<\mu(B)<\infty$ for all balls $B$.
We also fix a \emph{base point} $a \in X$
and let
\(
   |x|:=d(x,a).
\)

Following Bonk--Kleiner~\cite{BonkKleiner02}
we define the sphericalization of $X$ as follows.
Let $\Xdot=X\cup\{\binfty\}$.
Define $d_a,\dha:\Xdot \times \Xdot \to [0,\infty)$ by
\begin{equation*}      
  d_a(x,y) = d_a(y,x)=\begin{cases}
    \dfrac {d(x,y)}{(1+|x|)(1+|y|)},&\text{if }x,y\in X,\\[3mm]
    \dfrac{1}{1+|x|},&\text{if }x\in X \text{ and } y=\binfty,\\[3mm]
    0,&\text{if }x=y=\binfty,
                    \end{cases}
\end{equation*}
and
\[
   \dha(x,y)=\inf_{x=x_0,x_1,\dots,x_k=y}  \sum_{j=1}^k d_a(x_{j-1},x_{j}),
\]
where the infimum is taken over all finite sequences 
$x=x_0,x_1,\dots,x_k=y$ in $\Xhat$.
This makes $\dha$ into a metric on $\Xdot$, and $(\Xdot,\dha)$ 
is the \emph{sphericalization} of $(X,d)$.

As in our earlier paper~\cite{BBLi},
we equip $\Xhat$ with the 
measure 
\begin{equation*}   
d\muhq(x) = \frac{d\mu(x)}{(1+|x|)^q}
\quad \text{with }  \muhq(\{\binfty\})=0
\end{equation*}
for suitable $q>0$.
The choice $q=2p$ will preserve \p-harmonic functions.

The standard assumptions for the theory of \p-harmonic functions 
are that $X$ is complete and $\mu$ is a doubling measure
supporting a \p-Poincar\'e inequality.
To be able to exploit sphericalization for \p-harmonic functions
we need these properties
to hold also for the sphericalized space $(\Xhat,\dhat,\muhtp)$.
A large part of the paper is therefore devoted
to results showing when these 
properties are preserved.

The following result shows when the doubling condition 
is preserved.
(See Section~\ref{sect-doubling} for the definition of 
uniform perfectness.)

\begin{thm}   \label{thm-doubling-intro}
Assume that $X$ is uniformly perfect at the base point $a$ for radii $\ge1$,
and that $\mu$ is doubling on $X$.
Then $\muhq$ is doubling on $\Xhat$ if and only if there
is $0 <s <q$ such that 
\begin{equation}    \label{eq-dim-cond-at-a-intro}
\frac{\mu(B(a,r))}{\mu(B(a,R))} \simge \Bigl( \frac{r}{R} \Bigr)^s
\quad \text{whenever } 1\le r\le R<\infty.
\end{equation}
\end{thm}

For the preservation of \p-Poincar\'e inequalities 
we obtain the following result.
This is a special case of Theorem~\ref{thm-main-PI}, 
which contains  more details.

\begin{thm}\label{thm-main-PI-intro}
Assume that $X$ is complete and 
annularly connected for large radii around $a$, 
as in Definition~\ref{deff-ann-conn}.
Also assume that $\mu$ is doubling,  
satisfies~\eqref{eq-dim-cond-at-a-intro}
for  some $s>0$, and supports a \p-Poincar\'e  inequality with $p \ge 1$.

Then the sphericalization $(\Xhat,\dha,\muhq)$, with $q>s$, also supports a 
\p-Poincar\'e inequality.
\end{thm}

Metric properties of sphericalizations, such as (annular) quasiconvexity,
were studied in Bonk--Kleiner~\cite{BonkKleiner02}, Balogh--Buckley~\cite{BalBuc}
and Buckley--Herron--Xie~\cite{BuckleyHerronXie}.

Preservation of the doubling property, Ahlfors regularity and Poincar\'e inequalities
has been studied under various assumptions by 
Wildrick~\cite{Wi}, Li--Shan\-mu\-ga\-lin\-gam~\cite{LiShan}
and Durand-Cartagena--Li~\cite{DL1}, \cite{DL2},
but with other measures on $\Xhat$, 
so these results do not apply here.
In particular, their measures do not preserve \p-harmonic functions.

Sphericalization of  unbounded uniform metric spaces $(\Om,d)$, 
based on rectifiable curves as in \cite{BalBuc} and with measures 
tailored to preserve \p-harmonicity as in our earlier paper \cite{BBLi},
was considered in Gibara--Korte--Shan\-mu\-ga\-lin\-gam~\cite{GiKorSh} 
and Korte--Rogovin--Shan\-mu\-ga\-lin\-gam--Takala~\cite{KoRoShTa},
and used to solve a Dirichlet problem in $\Om$.
However, their assumptions and results are different from ours.
In particular, the assumption that $\Om$ is uniform (which is essential for the
proofs in \cite{GiKorSh} and \cite{KoRoShTa}) excludes many domains.
In our situation, the open set $\Om$ considered for the Dirichlet problem
is arbitrary, even though we impose some assumptions on the underlying space $X$.

The main reason for our choice of the measure $\muhq$ is that with
$q=2p$ and $p>1$, \p-harmonic functions are  preserved under sphericalization,
see Section~\ref{sect-pharm}.
When $X$ is  Ahlfors regular, this was 
used already in our earlier paper~\cite{BBLi}.
But even for  
Ahlfors regular spaces, the assumptions  in this paper
are weaker than in~\cite{BBLi}, see Remark~\ref{rmk-comp-BBLi}.
However, the main point is that we replace the
Ahlfors regularity assumption in~\cite{BBLi} by the much
weaker (and more natural) doubling property.
In particular, this makes it possible to include weighted $\R^n$ 
with \p-admissible weights, as in 
Heinonen--Kilpel\"ainen--Martio~\cite{HeKiMa}, into our considerations.
The doubling 
property and a \p-Poincar\'e inequality, together with completeness,
are  standard assumptions 
for the theory of \p-harmonic functions on metric spaces,
so these are natural assumptions for the results below to hold.

Another useful transformation, called uniformization,
 between unbounded and bounded spaces
was introduced and studied in Bonk--Heinonen--Koskela~\cite{BHK}.
It maps unbounded Gromov hyperbolic spaces to bounded uniform spaces. 
In Bj\"orn--Bj\"orn--Shan\-mu\-ga\-lin\-gam~\cite[Propositions~10.4 and 10.5]{unifPI}, 
it was shown that, when equipped with suitable measures, it also 
preserves \p-harmonic functions and can be used to obtain a Liouville type theorem in 
Gromov hyperbolic spaces.

In this paper, we apply sphericalization 
to obtain new results on solutions of the Dirichlet boundary value problem 
in unbounded open sets
and about  boundary regularity at $\binfty$.
We use the Perron method
and say that a boundary function
$f:\bdy \Om  \cup \{\binfty\} \to [-\infty,\infty]$ 
is \emph{resolutive}
if its lower and upper Perron solutions coincide,
in which case the common solution is denoted by $Pf$,
see Definition~\ref{def-Perron}.

The following resolutivity and invariance result
improves upon Propositions~11.6 and~11.7 in~Bj\"orn--Bj\"orn~\cite{BBglobal},
where a similar result was obtained under considerably
more restrictive conditions 
on $h$, in particular that $h$ is bounded.
In particular,  the Perron solution $Pf$ is the unique bounded
\p-harmonic function in $\Om$ satisfying both \eqref{eq-lim-qe-intro} and
$ \lim_{\Om \ni y \to \binfty} u(y)=f(\binfty)$.

Corresponding results in \p-parabolic spaces 
(as in Definition~\ref{def-p-par}) were
obtained by Hansevi~\cite[Theorem~7.8]{Hansevi2}.
Here $\Cp$ denotes the Sobolev capacity.
Theorem~\ref{thm-resol-C-intro} is a special case of Theorem~\ref{thm-resol-C}
and Corollary~\ref{cor-uniq-C}, which both use a more refined capacity.

\begin{thm} \label{thm-resol-C-intro}
Assume that the assumptions in Theorem~\ref{thm-main-PI-intro}
are satisfied with $p>1$ and $q=2p$.
Also assume that $X$ is \p-hyperbolic 
{\rm(}as in Definition~\ref{def-p-par}\/{\rm)} and 
that $\Om \subset X$ is   an unbounded open set.

Let $f \in C(\bdy \Om \cup \{\binfty\})$.
Assume that $h:\bdy \Om \to \eR$ vanishes  
$\Cp$-q.e.\ 
and let $h(\binfty)=0$.
Then both $f$ and $f+h$ are resolutive and $\Hp f = \Hp (f+h)$.

Moreover, if $u$ is a bounded
\p-harmonic function in $\Om$ such that
\begin{equation}   \label{eq-lim-qe-intro}
     \lim_{\Om \ni y \to x} u(y)=f(x)
     \quad \text{for $\Cp$-q.e.\ } x \in \bdry \Om
\end{equation}
and 
$ \lim_{\Om \ni y \to \binfty} u(y)=f(\binfty)$,
then $u=Pf$.
\end{thm}

The following theorem is a so-called trichotomy at $\binfty$.
Similar results for \emph{finite} boundary points in metric spaces were obtained in 
Bj\"orn~\cite[Theorem~2.1]{ABclass} (for bounded $\Om$) and 
Bj\"orn--Hansevi~\cite{BH2} (for unbounded $\Om$).
In a linear axiomatic setting in unweighted $\R^n$, a similar characterization
is due to Luke\v{s}--Mal\'y~\cite{LukMal}.

Note that $f \in C(\bdy \Om \cup \{\binfty\})$ is resolutive
by Theorem~\ref{thm-resol-C-intro} if $X$ is \p-hyperbolic,
and by Theorem~7.8 in Hansevi~\cite{Hansevi2} if $X$ is \p-parabolic.

\begin{thm} \label{thm-trich-intro}
Assume that the assumptions in Theorem~\ref{thm-main-PI-intro}
are satisfied with $p>1$ and $q=2p$.
Also assume that $\Om \subset X$ is   an unbounded open set,
and that $X$ is \p-hyperbolic or 
that $\Cp(X \setm \Om)>0$.

Then at least one of the following two situations happens at $\binfty$\textup:
\begin{enumerate}
\item \label{tr-lim-ex}
The limit $\displaystyle\lim_{\Om\ni x\to\binfty} Pf(x)$ exists for every 
$f\in C(\bdy \Om \cup \{\binfty\})$. 
\item \label{tr-lim=f}
There is a sequence $\Om\ni x_j\to\binfty$ such that 
\[
\lim_{j\to\infty} Pf(x_j) = f(\binfty) \quad \text{for every 
$f\in C(\bdy \Om \cup \{\binfty\})$. 
}
\]
\end{enumerate}
\end{thm}

The main point in Theorem~\ref{thm-trich-intro} is that it is
impossible for both properties \ref{tr-lim-ex}--\ref{tr-lim=f} to fail.
Note that the sequence in \ref{tr-lim=f} is independent of $f$.
If both \ref{tr-lim-ex} and \ref{tr-lim=f} hold, then $\binfty$ is \emph{regular}.
If only \ref{tr-lim-ex} holds, then $\binfty$ is called \emph{semiregular}, while
if only \ref{tr-lim=f} holds, then $\binfty$ is called \emph{strongly irregular}.
For finite boundary points, these latter two properties correspond
to the irregularity of $0$ in the punctured ball as in Zaremba~\cite{zaremba},
and to the Lebesgue spine as in Lebesgue~\cite{lebesgue1912}, respectively.

We also show that regularity of $\binfty$ can be characterized
using barriers in the following way.
In particular, condition~\ref{b-bbb-intro} below  says that regularity at $\binfty$
is a local property.

\begin{thm}  \label{thm-barrier+local-intro}
Assume that the assumptions in Theorem~\ref{thm-main-PI-intro}
are satisfied with $p>1$ and $q=2p$.
Also assume that $\Om \subset X$ is   an unbounded open set,
and 
that $X$ is \p-hyperbolic or that 
$\Cp(X \setm \Om)>0$.
Let $K \subset X$ be compact.

Then the following are equivalent\/\textup:
\begin{enumerate}
\item
$\binfty$ is regular with respect to $\Om$,
\item
there is a barrier at $\binfty$, i.e.\ a superharmonic function $u$ in $\Om$
such that
\[
\lim_{\Om\ni y\to \binfty} u(y)=0 \qquad \text{and} \qquad 
\liminf_{\Om\ni y\to x} u(y)>0 \quad \text{for every } x\in\bdy\Om.
\]
\item \label{b-bbb-intro}
$\binfty$ is regular with respect to $\Om \setm K$.
\end{enumerate}
\end{thm}

The barrier characterization
for \emph{finite} regular points in unbounded open sets 
is due to Bj\"orn--Hansevi~\cite[Theorem~6.2]{BH1}.
In \p-hyperbolic spaces all three conditions in 
Theorem~\ref{thm-barrier+local-intro}
are true (and thus also equivalent) by
Lemma~11.1 in Bj\"orn--Bj\"orn~\cite{BBglobal}.
On unweighted and weighted $\Rn$, 
Theorem~\ref{thm-barrier+local-intro} is due to 
Kilpel\"ainen~\cite[Theorem~1.5 and Remark~5.3]{Kilp89}
and Heinonen--Kilpel\"ainen--Martio~\cite[Theorem~9.8 and Corollary~9.15]{HeKiMa}.

The outline of the paper is as follows:
In Section~\ref{sect-prelim} we introduce some
notation from the analysis on metric spaces.
The sphericalization is introduced from a metric point of view
in Section~\ref{sect-sphericalization}.
Some key estimates for balls are also obtained here.

In Section~\ref{sect-doubling}, the measure $\muhq$ is studied
and Theorem~\ref{thm-doubling-intro} is deduced.
In Section~\ref{sect-PI} we turn to the preservation of
Poincar\'e inequalities and prove Theorem~\ref{thm-main-PI-intro}.

Section~\ref{sect-conseq-ug} is devoted
to how upper gradients, Newtonian spaces and capacities are transformed
by the sphericalization from $(X,d,\mu)$ to $(\Xhat,\dhat,\muhq)$.
This is then 
used in the next section
to show the preservation of \p-harmonicity when $q=2p$ and $p>1$.
In Section~\ref{sect-pharm} we also deduce 
Theorems~\ref{thm-resol-C-intro}--\ref{thm-barrier+local-intro}.

\begin{ack}
A.~B. and J.~B. were supported by the Swedish Research Council,
  grants 2020-04011 and 2022-04048, respectively.
We also thank the anonymous referee for a careful reading and useful suggestions.
\end{ack}

\section{Metric spaces and upper gradients}
\label{sect-prelim}

\emph{We assume throughout the paper
that $1 \le p<\infty$ and that
$X=(X,d)$ is a metric space.
Except for Section~\ref{sect-sphericalization}
we also assume that $X$ is 
equipped with  
a positive complete  Borel  measure $\mu$
such that $0<\mu(B)<\infty$ for all 
balls $B \subset X$.
Additional standing
assumptions are added at  the beginning of 
various sections.}

\medskip

The assumption $0 < \mu(B)<\infty$ implies that
$X$ is separable and Lindel\"of.
Proofs of the results in 
this section can be found in the monographs
Bj\"orn--Bj\"orn~\cite{BBbook} and
Heinonen--Koskela--Shan\-mu\-ga\-lin\-gam--Tyson~\cite{HKSTbook}.

The measure $\mu$  is \emph{doubling}
if  there is a \emph{doubling constant} $C_\mu>1$ such that
for all balls
$B=B(x_0,r):=\{x\in X: d(x,x_0)<r\}$ in~$X$,
\begin{equation*}
        0 < \mu(2B) \le C_\mu \mu(B) < \infty.
\end{equation*}
Here and elsewhere we write
$\lambda B=B(x_0,\lambda r)$.
All balls considered in this paper are assumed to be open,
unless said otherwise.

The following simple lemma 
will be used several times. 
A proof can be found e.g.\ in Bj\"orn--Bj\"orn~\cite[Lemma~3.6]{BBbook}.
Here and later
we write $a \simle b$ and $b \simge a$ 
if there is an implicit comparison constant $C>0$ such that $a \le Cb$, 
and $a \simeq b$ if $a \simle b \simle a$.
The implicit comparison constants are allowed to depend on the 
fixed data.
When needed, we will
explain the dependence in each case.

\begin{lem} \label{lem-comp-close-balls}
Assume that $\mu$ is a doubling measure.
Let $B=B(x,r)$ and $B'=B(x',r')$ be two balls such that
\[
d(x,x')\le cr
\quad \text{and} \quad  
\frac{r}{c}\le r'\le cr.
\]
Then $\mu(B) \simeq \mu(B')$ with comparison constants
depending only on $c \geq 1$ and $C_\mu$.
\end{lem}

The measure  $\mu$ is \emph{Ahlfors $Q$-regular}, $Q>0$, if
\[ 
     \mu(B(x,r)) \simeq r^Q
\quad \text{when } r \le 2 \diam X.
\]

$X$ is \emph{proper} if every closed bounded subset is compact.
If $\mu$ is doubling, then $X$ is proper 
if and only if it is complete.

A \emph{curve} is a continuous mapping from an interval,
and a \emph{rectifiable} curve is a curve with finite length.
A rectifiable curve can be parametrized by arc length $ds$.
A property holds for \emph{\p-almost every curve}
if the curve family $\Ga$ of rectifiable curves
for which it fails has zero \p-modulus,
i.e.\ there is $\rho\in L^p(X)$ such that
$\int_\ga \rho\,ds=\infty$ for every $\ga\in\Ga$.
Following Hei\-no\-nen--Koskela~\cite{HeKo98} and Koskela--MacManus~\cite{KoMc}
we define upper gradients and  \p-weak upper gradients as follows.

\begin{deff} \label{deff-ug}
A Borel function $g : X \to [0,\infty]$  is an \emph{upper gradient} 
of a function $u: X \to \eR:=[-\infty,\infty]$
if for all  nonconstant rectifiable curves  
$\gamma : [0,1] \to X$,
\begin{equation} \label{ug-cond}
|u(\gamma(0)) - u(\gamma(1))| \le \int_{\gamma} g\,ds,
\end{equation}
where the left-hand side is interpreted as
$\infty$ whenever at least one of the terms therein is infinite.
If $g: X \to [0,\infty]$ is measurable 
and \eqref{ug-cond} holds for \p-almost every nonconstant rectifiable curve,
then $g$ is a \emph{\p-weak upper gradient} of~$u$. 
\end{deff}

If $g \in \Lploc(X)$ is a \p-weak upper gradient of $u$,
then one can find a sequence $\{g_j\}_{j=1}^\infty$
of upper gradients of $u$ such that $\|g_j-g\|_{L^p(X)}  \to 0$ as $j \to \infty$.
If $u$ has an upper gradient in $\Lploc(X)$, then
it has an a.e.\ unique \emph{minimal \p-weak upper gradient} $g_u \in \Lploc(X)$
in the sense that $g_u \le g$ a.e.\ for every \p-weak upper gradient 
$g \in \Lploc(X)$ of $u$.

Together with the doubling property defined above, the following
Poincar\'e inequality is often a standard assumption on metric spaces.

\begin{deff} \label{def.PI.}
We say that $X$ (or $\mu$) supports a \emph{\p-Poincar\'e inequality}
if
there exist constants $\CPI>0$ and $\lambda \ge 1$
such that for all balls $B \subset X$,
all integrable functions $u$ on $X$ and all 
upper gradients $g$ of $u$,
\[ 
        \vint_{B} |u-u_B| \,d\mu
        \le \CPI  r_B \biggl( \vint_{\lambda B} g^{p} \,d\mu \biggr)^{1/p},
\] 
where $u_B :=\vint_B u \,d\mu := \int_B u\, d\mu/\mu(B)$
and $r_B$ is the radius of the ball $B$.
The constant $\la$ is called the \emph{dilation constant}.
\end{deff}

Note that one can equivalently require
that $u$ is integrable on $\la B$ and that $g$ 
is an upper gradient (or \p-weak upper gradient) of $u$ on $\la B$,
see the proof of Theorem~8.1.53 in
Heinonen--Koskela--Shan\-mu\-ga\-lin\-gam--Tyson~\cite{HKSTbook}.
The constants $\CPI$
and $\la$ remain the same.
If $X$ supports a \p-Poincar\'e inequality then it is connected.

The space $X$  is \emph{$L$-quasiconvex}
if for every $x,y \in X$ there is a curve
$\ga$ from $x$ to $u$ with length $l_\ga \le L d(x,y)$.
The following well-known result is due to Semmes, 
for a proof
see e.g.\ 
\cite[Theorem~4.32]{BBbook} or \cite[Theorem~8.3.2]{HKSTbook}.

\begin{thm}          \label{thm-PI-imp-quasiconvex}
If $X$ is complete  and
supports a \p-Poincar\'e  inequality 
and $\mu$ is doubling,
then $X$ is $L$-quasiconvex,
with $L$ only depending on $C_\mu$, $\CPI$ and $\la$.
\end{thm}

The Newtonian Sobolev space on a metric space $X$ was
introduced as follows by 
Shan\-mu\-ga\-lin\-gam~\cite{Sh-rev}.

\begin{deff} \label{deff-Np}
For measurable $u$, let
\[
        \|u\|_{\Np(X)} = \biggl( \int_X |u|^p \, d\mu
                + \inf_g  \int_X g^p \, d\mu \biggr)^{1/p},
\]
where the infimum is taken over all upper gradients of $u$.
The \emph{Newtonian space} on $X$ is
\[
        \Np (X) = \{u: \|u\|_{\Np(X)} <\infty \}.
\]
\end{deff}

The space $\Np(X)/{\sim}$, where  $u \sim v$ if and only if $\|u-v\|_{\Np(X)}=0$,
is a Banach space and a lattice.
We also define the  \emph{Dirichlet space} 
\[
   \Dp(X)=\{u : u \text{ is measurable, finite
a.e.\ and  has an upper gradient
     in }   L^p(X)\}.
\]
This definition deviates from the definition in e.g.\ 
\cite[Definition~1.54]{BBbook} in that it requires the functions
to be finite a.e., 
see the comments after Theorem~\ref{thm-quasicont}.
In this paper, it is convenient to
assume that functions in $\Np$ and $\Dp$
 are defined everywhere (with values in $\eR$),
not just up to an equivalence class in the corresponding function space.

We say  that $u \in \Nploc(X)$ if
for every $x \in X$ there exists a ball $B_x\ni x$ such that
$u \in \Np(B_x)$. 
The space $\Dploc(X)$ is defined similarly.
If $u,v \in \Dploc(X)$,
then $g_u=g_v$ a.e.\ in $\{x \in X : u(x)=v(x)\}$.
In particular, for $c \in \R$ we have $g_{\min\{u,c\}}=g_u \chione_{\{u < c\}}$ 
a.e.\ in $X$,
where $\chione$ denotes the characteristic function.
For a measurable set $E\subset X$, the 
spaces $\Np(E)$, $\Nploc(E)$, $\Dp(E)$ and $\Dploc(E)$  are defined by
considering $(E,d|_E,\mu|_E)$ as a metric space in its own right.

The  \emph{Sobolev capacity} of an arbitrary set $E\subset X$ is
\begin{equation*} 
\Cp(E) =\inf_{u}\|u\|_{\Np(X)}^p,
\end{equation*}
where the infimum is taken over all $u \in \Np(X)$ such that
$u\geq 1$ on $E$.
It is easy to see that the Sobolev capacity is countably subadditive.
For further properties, see 
e.g.\ \cite{BBbook} and~\cite{HKSTbook}.

A property holds \emph{quasieverywhere} (q.e.)\
if the set of points  for which it fails
has Sobolev capacity zero.
The Sobolev capacity is the correct gauge
for distinguishing between two Newtonian functions, namely
$v \sim u$ if and only if $v=u$ q.e.
Moreover, if $u,v \in \Np(X)$ and $u=v$ a.e., then $u=v$ q.e.

A function $u :X \to \eR$ is \emph{quasicontinuous}
if for every $\eps>0$ there is an \emph{open} set $G$
with $\Cp(G)<\eps$ such that $u|_{X \setm G}$ is continuous.
Here and elsewhere, continuous functions are
always assumed to be real-valued, while 
quasicontinuous functions are
allowed to be $\eR$-valued.

That Newtonian functions are quasicontinuous was first
shown by 
Bj\"orn--Bj\"orn--Shan\-mu\-ga\-lin\-gam~\cite{BBS5},
assuming that $X$ is complete and that $\mu$ is a doubling measure
supporting a \p-Poincar\'e inequality.
These assumptions have later been substantially weakened, first by
Ambrosio~et al.~\cite{AmbCD},~\cite{AmbGS}
and more recently by Eriksson-Bique--Poggi-Corradini~\cite[Theorem~1.3]{EB-PC},
who only require that $X$ is locally complete.
The following theorem 
shows how to extend the result from~\cite{EB-PC}
to functions in $\Dp$.
This will be useful
in Section~\ref{sect-pharm}.

\begin{thm} \label{thm-quasicont}
Let $u \in \Dploc(X)$.
Then the following hold\/\textup{:}
\begin{enumerate}
\item \label{q-Cp}
$\Cp(\{x:|u(x)|=\infty\})=0$.
\item \label{q-qcont}
If $X$ is locally complete, then $u$ is quasicontinuous.
\end{enumerate}
\end{thm}

In e.g.\ \cite[Definition~1.54]{BBbook}, 
functions $u \in \Dp(X)$ were not required to be finite a.e.
However, if
there are no rectifiable curves in $X$, then
such a definition would allow
$u \equiv \infty \in \Dp(X)$, which   
is not a quasicontinuous function.
In fact, no function $u$ with $\mu(\{x:|u(x)|=\infty\})>0$ can be
quasicontinuous, so 
the implicit condition that $u$ is finite a.e.\ 
is essential
in Theorem~\ref{thm-quasicont}, unlike for functions in $\Np$
where it is automatically satisfied.

 \begin{proof}
 As $X$ is Lindel\"of, 
 it can be covered by a countable collection 
 $\{B_j\}_{j=1}^\infty$ of balls such that $u \in \Dp(B_j)$ for each $j$.

 \ref{q-Cp}
 It follows from Theorem~1.56 and Proposition~2.27 in~\cite{BBbook}
 that 
 $u$ is absolutely continuous on \p-almost every rectifiable curve.
 Since $u$ is finite a.e., by definition,
 \ref{q-Cp} now follows from Corollary~1.70 in~\cite{BBbook}.

 \ref{q-qcont}
 Let $\eps>0$ and let
 \[
         u_k = \min\{k, \max\{u, -k\}\}, 
 \quad k =1,2,\dots,
 \]
 be the truncations of $u$ at the levels $\pm k$.
 Then $g_u$ is a \p-weak upper gradient also of $u_k$.
 Hence $u_k \in \Np(B_j)$ for each $j$.
 As $X$ is locally complete, so is $B_j$ and 
 therefore $u_k$ is quasicontinuous in $B_j$ 
with respect to the capacity
 $\CpBj$ associated with $\Np(B_j)$,
 by Theorem~1.3 in 
 Eriksson-Bique--Poggi-Corradini~\cite{EB-PC}.
Since the balls $B_j$ are open, quasicontinuity with respect to 
 the capacities $\CpBj$ and $\Cp$ is the same, cf.\ Proposition~3.4 in 
 Bj\"orn--Bj\"orn--Mal\'y~\cite{BBHonza}.
 Hence,
 there is an open set $G_{j,k}$ such that $\Cp(G_{j,k}) < 2^{-j-k}\eps$
 and $u_k|_{B_j \setm G_{j,k}}$ is continuous. 

 Moreover, by \ref{q-Cp} and Theorem~1.5 in~\cite{EB-PC}
 (which says that $\Cp$ is an ``outer capacity''),
 there is an open set $G' \supset \{x:|u(x)|=\infty\})$ such that
 $\Cp(G') <\eps$.
 Let $G=G' \cup \bigcup_{j,k=1}^\infty G_{j,k}$.
 Then $\Cp(G)<2\eps$.

 Next, let $x \in X \setm G$.
 Then $u(x) \in \R$ and there are $j$ and $k$
 such that $x \in B_j$ and $|u(x)| < k$.
 Since $u_k|_{B_j \setm G}$ is continuous at $x$, also
 $u|_{X \setm G}$ must be  continuous at $x$.
 Therefore $u|_{X \setm G}$ is continuous.
 \end{proof}

\section{Sphericalization}
\label{sect-sphericalization}

\emph{In this section we assume that $X$ is an unbounded metric  space.}

\medskip

In this section we do not need any measure on $X$.
Moreover, the  results in this section
hold also when $X$ is nonseparable.
From now on, we also fix a \emph{base point} $a \in X$.
To simplify the notation we write 
\[
   |x|:=d(x,a).
\]

Following Bonk--Kleiner~\cite{BonkKleiner02}
we define the sphericalization of $X$ as follows.
Let $\Xdot=X\cup\{\binfty\}$.
Define $d_a,\dha:\Xdot \times \Xdot \to [0,\infty)$ by
\begin{equation*}      
  d_a(x,y) = d_a(y,x)=\begin{cases}
    \dfrac {d(x,y)}{(1+|x|)(1+|y|)},&\text{if }x,y\in X,\\[3mm]
    \dfrac{1}{1+|x|},&\text{if }x\in X \text{ and } y=\binfty,\\[3mm]
    0,&\text{if }x=y=\binfty,
                    \end{cases}
\end{equation*}
and
\[
   \dha(x,y)=\inf_{x=x_0,x_1,\dots,x_k=y}  \sum_{j=1}^k d_a(x_{j-1},x_{j}),
\]
where the infimum is taken over all finite sequences 
$x=x_0,x_1,\dots,x_k=y$.
This makes $\dha$ into a metric on $\Xdot$, and $(\Xdot,\dha)$ 
is the \emph{sphericalization} of $(X,d)$.
Moreover, 
\begin{equation}   \label{eq-comp-da-dha}
\tfrac{1}{4} d_a(x,y) \le \dha(x,y) \le d_a(x,y),
\end{equation}
see 
the proof of Lemma~2.2 in~\cite{BonkKleiner02}.
The sphericalization
$(\Xhat,\dhat)$ is compact if and only if $X$ is proper, 
in which case $(\Xhat,\dhat)$ is 
topologically the one-point compactification of $(X,d)$.

Note that $d_a$ is in general not a metric since the triangle inequality
may fail for it.
However, because of~\eqref{eq-comp-da-dha}, it is a quasimetric and 
satisfies 
\[
d_a(x,y)\le 4(d_a(x,z)+d_a(y,z)).
\]
Moreover, it is easily verified using elementary calculations and
the triangle inequality for $d(x,y)$ that the special triangle inequality 
\begin{equation} \label{eq-triang-for-a-infty}
d_a(x,y) \le d_a(x,z) + d_a(z,y)
\quad \text{holds whenever } \{x,y,z\} \cap \{a,\binfty\} \ne\emptyset.
\end{equation}
For example, with $y=\binfty$,
\[
d_a(x,\binfty) = \frac{1+|z|}{(1+|x|)(1+|z|)} \le \frac{1+|x|+d(x,z)}{(1+|x|)(1+|z|)}
= d_a(z,\binfty)  + d_a(x,z).
\]
(Here and later we let $|\binfty|=\infty$, with appropriate
interpretation when it appears in quotients like those above.)
Hence, 
if $x=x_0,x_1,\dots,x_k=\binfty$, then
\begin{align*}
d_a(x,\binfty) & \osr{eq-triang-for-a-infty}{\le} d_a(x_0,x_1)+d_a(x_1,\binfty) \\
&\osr{eq-triang-for-a-infty}{\le} d_a(x_0,x_1)+d_a(x_1,x_2)+d_a(x_2,\binfty)
\le \dots \le \sum_{j=1}^k d_a(x_{j-1},x_{j}),
\end{align*}
from which it follows that 
\begin{equation} \label{dhat=da-infty}
\dhat(x,\binfty)=d_a(x,\binfty)=\frac{1}{1+|x|}.
\end{equation}
Similarly 
\begin{equation} \label{dhat=da-a}
\dhat(x,a)=d_a(x,a)=\frac{|x|}{1+|x|}.
\end{equation}

Note that (with $|x|=d(x,a)$ as before)
\[
    \dha(x,y) \le d_a(x,y) 
    = \frac {d(x,y)}{(1+|x|)(1+|y|)}
    \le  \frac {|x|+|y|}{(1+|x|)(1+|y|)}
    < 1,
\quad \text{if } x,y \in X.
\]
Since also $\dhat(a,\binfty)\osr{dhat=da-infty}{=}d_a(a,\binfty)=1$,
we conclude that the
sphericalization $\Xdot$ is always bounded and
\begin{equation} \label{eq-diam}
\diam_{\dha}\Xdot= \diam_{d_a}\Xdot =1.
\end{equation}

We will denote balls in $(\Xdot,\dha)$ by $\Bhat(x,r)$,
while balls in $(X,d)$ are denoted by $B(x,r)$ as before.
The ``balls'' in $\Xdot$ with respect to $d_a$ will be denoted 
$B_a(x,r)=\{y\in\Xdot: d_a(x,y)<r\}$.
Clearly, \eqref{eq-comp-da-dha} implies that
\begin{equation}  \label{eq-comp-balls-da-dhat}
\Bhat(x,\tfrac14 r) 
\osr{eq-comp-da-dha}{\subset} B_a(x,r) 
\osr{eq-comp-da-dha}{\subset} \Bhat(x,r). 
\end{equation}

To further simplify the notation  we will also write
\begin{equation} \label{eq-dha(x)}
   \dha(x):=\dhat(x,\binfty)\osr{dhat=da-infty}{=}\frac{1}{1+|x|} 
   \osr{dhat=da-a}{=} 1-\dhat(x,a).
\end{equation}
These identities will be used extensively.

In the following lemma it will be convenient to use
the closed balls $\Bbar(x,\rho)=\{y \in X : d(y,x) \le \rho\}$
with $\rho \in\R$.
In other situations, radii are implicitly assumed to be positive.

\begin{lem}   \label{lem-ball-est}
We have 
\begin{equation} \label{eq-Bhat-infty}
\Bhat(\binfty,r) = \Xhat \setm \Bbar\biggl(a,\frac{1}{r}-1\biggr).
\end{equation}
Moreover, for $x\in X$,
\begin{alignat}{2}
   \label{eq-incl-small-r}
B\biggl(x,\frac{3r}{4\dhat(x)^2}\biggr) & \subset
B_a(x,r) \subset B\biggl(x,\frac{3r}{2\dhat(x)^2}\biggr), 
&\quad& \text{if } r\le\tfrac13\dhat(x), \\
\label{eq-incl-large-r}
B\biggl(x,\frac{1}{4\dhat(x)}\biggr) &\subset
B_a(x,r) \subset 
\Bhat(x,r) \subset 
\Xhat\setm \Bbar\biggl(a,\frac{1}{4r}-1\biggr),
&\quad&
\text{if } \tfrac13\dhat(x) \le r\le 1.
\end{alignat}
\end{lem}

\begin{proof}
The 
identity \eqref{eq-Bhat-infty}
follows directly from \eqref{eq-dha(x)}.
To prove~\eqref{eq-incl-small-r}, let
\(
r\le\tfrac13\dhat(x)  
\)
and assume that $y\in B_a(x,r)$.
Then
\[
d(x,y) < r(1+|x|)(1+|y|) \le r(1+|x|)(1+|x|+d(x,y))
\osr{eq-dha(x)}{=} \frac{r}{\dhat(x)^2}+\frac{rd(x,y)}{\dhat(x)},
\]
from which we obtain
\[
d(x,y) \biggl(1-\frac{r}{\dhat(x)}\biggr) < \frac{r}{\dhat(x)^2}.
\]
Since $1-r/\dhat(x)\ge\tfrac23$, we conclude that
\[
d(x,y) <  \frac{3r}{2\dhat(x)^2},
\]
i.e.\ the second inclusion in~\eqref{eq-incl-small-r} holds.

Similarly, for the first inclusion, assume that $d_a(x,y)\ge r$.
Then
\[
d(x,y) \ge r(1+|x|)(1+|y|) \ge r(1+|x|)(1+|x|-d(x,y))
\osr{eq-dha(x)}{=} \frac{r}{\dhat(x)^2}-\frac{rd(x,y)}{\dhat(x)},
\]
which yields
\[
d(x,y) \biggl(1+\frac{r}{\dhat(x)}\biggr) \ge \frac{r}{\dhat(x)^2}.
\]
Since $1+r/\dhat(x)\le\tfrac43$, 
we obtain that $d(x,y)\ge \tfrac34 r/\dha(x)^2$
and the first inclusion in~\eqref{eq-incl-small-r} follows.

Now assume that $r\ge \tfrac13\dhat(x)$.
To prove the first inclusion in~\eqref{eq-incl-large-r}, 
note that by~\eqref{eq-incl-small-r},
\[
B_a(x,r) \supset B_a(x,\tfrac13 \dhat(x)) 
\osr{eq-incl-small-r}{\supset}
B\biggl(x,\frac{1}{4\dhat(x)}\biggr). 
\] 
The second inclusion follows from \eqref{eq-comp-balls-da-dhat},
while the
third inclusion follows from
\eqref{eq-Bhat-infty} since 
$\Bhat(x,r)\subset\Bhat(\binfty,4r)$.
\end{proof}

If $\ga :[0,1] \to X$ is a (not necessarily rectifiable) curve, then 
it is quite easy 
(cf.\ Li--Shan\-mu\-ga\-lin\-gam~\cite[(2.11)]{LiShan})
to see that the arc lengths $ds$ and $d\sh$ with respect to $d$ and 
$\dha$, respectively,
are related by
\begin{equation} \label{eq-dsha}
   d\sh(x)= \frac{ds(x)}{(1+|x|)^2}
\osr{eq-dha(x)}{=} \dhat(x)^2\, ds(x).
\end{equation}
As $\ga([0,1])$ is compact it follows that $\ga$ is rectifiable with 
respect to $d$ if and only if it is rectifiable with respect to $\dha$.
The following relation between upper gradients follows directly from \eqref{eq-dsha}
and \eqref{ug-cond}.
(The definition of upper gradient in Definition~\ref{deff-ug} does not
require any measure on $X$.
Corollary~\ref{cor-pwug} extends this identity to \p-weak upper gradients,
which also involves the measure.)
Note that the formulation in \cite{LiShan} is 
for Lipschitz functions on $\Xhat$,
but their proof directly covers our case.

\begin{lem} \label{lem-upper-grad}
\textup{(\cite[Lemma~3.4]{LiShan})}
Let $E \subset X$.
A Borel function $g:E \to [0,\infty]$ is an 
upper gradient of  $u:E \to \eR$ with respect to $d$  if and only if 
\[
    \gh(x)=g(x)(1+|x|)^2
\osr{eq-dha(x)}{=} \frac{g(x)}{\dhat(x)^2}, 
    \quad  x \in E,
\]
is an upper gradient of $u$ with respect to $\dha$.
\end{lem}

\section{Doubling measures and sphericalization}
\label{sect-doubling}

\emph{As in Section~\ref{sect-sphericalization}, we
assume in this section that $X$ is unbounded.
Also recall the standing assumptions from 
the beginning of Section~\ref{sect-prelim}.}

\medskip

In Li--Shan\-mu\-ga\-lin\-gam~\cite{LiShan}, 
the sphericalization $(\Xdot,\dha)$ was equipped
with the measure 
$\mut$ (called $\mu_a$ in~\cite{LiShan})
defined through
\[
d\mut(x)= \frac{d\mu(x)}{\mu(B(a,1+|x|))^2}
\quad \text{with }  \mut(\{\binfty\})=0.
\]
It is always true that $\mut(\Xdot) < \infty$, by
Proposition~3.1 in Bj\"orn--Bj\"orn--Li~\cite{BBLi}.
In this paper we instead consider 
\begin{equation}    \label{def-muha}
d\muhq(x) := \frac{d\mu(x)}{(1+|x|)^q}
\osr{eq-dha(x)}{=}\dhat(x)^q\, d\mu(x)
\quad \text{with }  \muhq(\{\binfty\})=0.
\end{equation}
for suitable $q>0$.
Observe that measurability is the same with respect to $\mu$ and $\muhq$.
The measure $\muhq$ was considered in~\cite{BBLi}
under more restricted assumptions than here.

Instead of connectedness it is sometimes enough to assume that 
the space $X$ is 
\emph{uniformly perfect}, 
i.e.\ there is a constant $\ka>1$ such that 
\begin{equation*} 
  B(x,\ka r) \setm B(x,r) \ne \emptyset
\quad \text{whenever } x \in X \text{ and } r>0.
\end{equation*}  
In this section the following pointwise version will be enough.
The space $X$ is \emph{uniformly perfect at $a$ for radii $\ge 1$}
with constant $\ka >1$ if 
\begin{equation*} 
  B(a,\ka r) \setm B(a,r) \ne \emptyset
\quad \text{whenever } r\ge 1.
\end{equation*}  
It is easy to see that the constant $1$ in $r\ge1$ can equivalently 
be replaced by any other positive number, provided that $\ka$ is allowed to change.
Similarly, the condition holds equivalently   at any other point
$x \in X$, again provided that $\ka$ is allowed to change.
This pointwise condition may have first been  used
in~Bj\"orn--Bj\"orn--Korte--Rogovin--Takala~\cite{BBKRT}.
See Heinonen~\cite[Chapter~11]{Heinonen} 
for further discussion and history 
of uniform perfectness.

\begin{lem}   \label{lem-meas-est}
Assume that $X$ is uniformly perfect at the base point $a$ for radii $\ge1$
with constant $\ka$.
Also assume that $\mu$ is doubling on $X$ and
that for some $s,C_s>0$,
\begin{equation}    \label{eq-dim-cond-at-a}
\frac{\mu(B(a,r))}{\mu(B(a,R))} \ge C_s\Bigl( \frac{r}{R} \Bigr)^s
\quad \text{whenever } 1\le r\le R<\infty.
\end{equation}

If $q>s$, then the measure $\muhq$, defined by~\eqref{def-muha}, satisfies
for $x \in \Xhat$,
\begin{equation}  \label{eq-comp-muh-mu}
  \muhq(B_a(x,r))\simeq \begin{cases}
    \displaystyle \dhat(x)^q  
     \mu\biggl( B\biggl(x,\frac{r}{\dhat(x)^2} \biggr) \biggr),
            &\text{if } r\le\tfrac13\dhat(x),  \\[4mm]  
     r^q \mu(B(a,1/r)),
            &\text{if } \tfrac16\dhat(x)\le r\le 2. 
                    \end{cases} 
\end{equation} 
The comparison constants in the first formula in~\eqref{eq-comp-muh-mu}
depend only on $C_\mu$, while in the second formula they
depend only on $C_\mu$, $\ka$, $C_s$ and $q-s$.
\end{lem}

It is well known that doubling measures 
in uniformly perfect spaces
always satisfy~\eqref{eq-dim-cond-at-a} 
with some $s,C_s>0$ (independent of the base point $a$).

\begin{proof}[Proof of Lemma~\ref{lem-meas-est}]
For the first formula in \eqref{eq-comp-muh-mu},
let  $r\le \tfrac13 \dhat(x)$. 
Then,
for all $y\in B_a(x,r)$,
\begin{equation*}
\frac{1}{\dhat(y)}\osr{eq-dha(x)}{=}
 1+|y| \le 1+|x| + d(x,y) 
\os{\eqref{eq-dha(x)},\eqref{eq-incl-small-r}}{<}
 \frac{1}{\dhat(x)} + \frac{3r}{2\dhat(x)^2}
\le  \frac{3}{2\dhat(x)}
\end{equation*}
and similarly
\begin{equation*}
\frac{1}{\dhat(y)}\osr{eq-dha(x)}{=}
1+|y| \ge 1+|x| - d(x,y) 
\os{\eqref{eq-dha(x)},\eqref{eq-incl-small-r}}{>}
\frac{1}{\dhat(x)} - \frac{3r}{2\dhat(x)^2}
\ge  \frac{1}{2\dhat(x)},
\end{equation*}
i.e.\ 
$\dhat(y) \simeq \dhat(x)$ 
in~\eqref{def-muha} for all $y\in B_a(x,r)$.
The inclusions~\eqref{eq-incl-small-r} and the doubling property of $\mu$
then immediately imply that 
\[
\muhq(B_a(x,r)) 
\simeq  \dhat(x)^q \mu\biggl( B\biggl(x,\frac{r}{\dhat(x)^2} \biggr) \biggr),
\]
which proves the first formula in~\eqref{eq-comp-muh-mu}.

For the second formula in~\eqref{eq-comp-muh-mu}
(with $\tfrac16\dhat(x)\le r\le 2$), we 
split the proof into several cases.

{\bf Case 1:} $x \in X$.
Then by~\eqref{eq-incl-large-r} (applied to $B_a(x,2r)$), 
\begin{align}   \label{eq-est-B-a}
\muhq(B_a(x,r)) &\le 
\muhq(B_a(x,2r)) \le 
\int_{1+|y|>1/8r} \frac{d\mu(y)}{(1+|y|)^q} \\
   &\le \sum_{j=1}^\infty \biggl( \frac{2^{j-4}}{r} \biggr)^{-q} 
      \mu (\{ y\in X: 1+|y|< 2^j/r \}).
    \nonumber
\end{align}
The assumption~\eqref{eq-dim-cond-at-a} with $r$ and $R$ replaced 
by $2/r\ge 1$ and $2^j/r$, respectively, yields
\[
\mu (\{ y\in X: 1+|y|< 2^j/r \})
  \le \mu(B(a,2^j/r))
\osr{eq-dim-cond-at-a}{\le}
\frac{2^{(j-1)s}}{C_s} \mu(B(a,2/r)).
\]
As $q>s$, inserting this into~\eqref{eq-est-B-a} 
and summing the geometric series implies that
\begin{equation*} 
\muhq(B_a(x,r))\simle r^q \mu(B(a,2/r))\simeq r^q \mu(B(a,1/r)),
\end{equation*}
which gives the upper bound in~\eqref{eq-comp-muh-mu}.
For the lower bound, there are two subcases.

{\bf Subcase 1a:} $\tfrac16\dhat(x)\le r\le 2\dhat(x)$.
In this case, the 
(already proved) first formula in~\eqref{eq-comp-muh-mu}
and  the doubling property of $\mu$ yield
\begin{align*}
\muhq(B_a(x,r)) &\ge \muhq(B_a(x,\tfrac16 \dhat(x))) \nonumber\\
&\osr{eq-comp-muh-mu}{\simeq} \dhat(x)^q \mu\biggl(B\biggl(x,\frac{1}{6\dhat(x)} \biggr) \biggr) 
\simeq \dhat(x)^q \mu(B(a,1/r)).  
\end{align*}
where the last comparison is guaranteed by Lemma~\ref{lem-comp-close-balls},
since 
$ d(x,a) = |x| \le 1/\dhat(x)\simeq 1/r$. 
This gives the lower bound in~\eqref{eq-comp-muh-mu}.

{\bf Subcase 1b:} $2\dhat(x) \le r \le2$.
If $\dhat(y) < \tfrac12 r$, then by the special triangle
inequality~\eqref{eq-triang-for-a-infty}, together with \eqref{dhat=da-infty},
\[
d_a(x,y) \osr{eq-triang-for-a-infty}{\le} 
d_a(x,\binfty) + d_a(y,\binfty)
\osr{dhat=da-infty}{=} \dhat(x) + \dhat(y) < r.
\]
Hence,
\begin{equation} \label{eq-6}
 B_a(x,r)\supset B_a(\binfty,\tfrac12 r)
\osr{eq-Bhat-infty}{=} \Xhat \setm \Bbar\biggl(a,\frac2r -1\biggr)
\supset X \setm B\biggl(a,\frac2r\biggr).
\end{equation}
Since $X$ is unbounded and uniformly perfect at $a$ for radii $\ge1$,
there is $z\in X$ such that 
\[
\frac3r \le  |z| < \frac{3\ka}{r}.
\]
For all
\begin{equation} \label{eq-7}
z'\in B(z,1/r) \subset X \setm B(a,2/r) \osr{eq-6}{\subset} B_a(x,r),
\end{equation}
we then have 
in \eqref{def-muha} that
\[
\frac{2}{r}  \le |z| - \frac1r
< 1+ |z'| \le 1+ |z| + \frac1r < 
\frac{6\ka}{r},
\]
using also $r\le2$    
in the last step.
Thus $d\muhq \simeq r^q\,d\mu$ in $B(z,1/r)$, which
implies that
\[
\muhq(B_a(x,r)) \osr{eq-7}{\ge}
\muhq(B(z,1/r)) \simeq
r^q \mu(B(z,1/r))
\simeq  r^q \mu(B(a,1/r)),
\]
by Lemma~\ref{lem-comp-close-balls} in the last step. This
gives the lower bound in~\eqref{eq-comp-muh-mu} and
concludes the proof of Case~1.

{\bf Case 2:}
$x=\binfty$.
As $X$ is unbounded, there is $y\in X$ with $\dhat(y) \le \tfrac12 r$.
Since $\mu$ is doubling, it then follows from 
\eqref{eq-comp-muh-mu} in Case~1, together with 
\eqref{eq-comp-balls-da-dhat}
that
\begin{align*}
\muhq(\Bhat(\binfty,r)) 
& \ge \muhq(\Bhat(y,\tfrac12 r)) 
\osr{eq-comp-balls-da-dhat}{\ge}  \muhq(B_a(y,\tfrac12 r))
\osr{eq-comp-muh-mu}{\simeq} r^q \mu(B(a,1/r)),
\\
\muhq(\Bhat(\binfty,r)) 
   & \le \muhq(\Bhat(y,\tfrac32 r)) 
\osr{eq-comp-balls-da-dhat}{\le}  \muhq(B_a(y,6 r))
   \osr{eq-comp-muh-mu}{\simeq} r^q \mu(B(a,1/r)).
\qedhere
\end{align*}
\end{proof}

\begin{cor}  \label{cor-muha-doubl}
Assume that the assumptions in Lemma~\ref{lem-meas-est} are satisfied and that $q>s$.
Then $\muhq$ is doubling on $\Xhat$ and the doubling constant
$C_\muhq$ depends only on 
$C_\mu$, $\ka$, $C_s$ and $q-s$. 
\end{cor}

\begin{proof}
Because of~\eqref{eq-diam} and~\eqref{eq-comp-balls-da-dhat},
it suffices to show that
\[
\muhq(B_a(x,2r)) \simle \muhq(B_a(x,r))
\quad \text{when $r \le 1$}.
\]
If
$r\le \tfrac16 \dhat(x)$, then 
by
Lemma~\ref{lem-meas-est} and the doubling property of~$\mu$,
\begin{align*}
\muhq(B_a(x,2r)) 
&\osr{eq-comp-muh-mu}{\simeq} \dhat(x)^q  \mu\biggl( B\biggl(x,\frac{2r}{\dhat(x)^2} \biggr) \biggr)
\\
&\simeq \dhat(x)^q  \mu\biggl( B\biggl(x,\frac{r}{\dhat(x)^2} \biggr) \biggr)
\osr{eq-comp-muh-mu}{\simeq} \muhq(B_a(x,r)).
\end{align*}
On the other hand, if 
$\tfrac16 \dhat(x) \le r\le 1$
then, by Lemma~\ref{lem-meas-est} again,
\[
\muhq(B_a(x,2r)) 
\osr{eq-comp-muh-mu}{\simeq} (2r)^q \mu(B(a,1/2r)) 
\simeq r^q \mu(B(a,1/r)) 
\osr{eq-comp-muh-mu}{\simeq} \muhq(B_a(x,r)).
\qedhere
\]
\end{proof}

\begin{remark} \label{rmk-Bhat}
Because of the doubling property of $\mu$, the constants $\tfrac13$
and $\tfrac16$
in~\eqref{eq-comp-muh-mu} can be replaced by any other positive constant,
at the cost of changing the comparison constants.
Also, because of~\eqref{eq-comp-balls-da-dhat} and the doubling property of $\muhq$, we have 
$\muhq(\Bhat(x,r)) \simeq \muhq(B_a(x,r))$, so the estimates in 
Lemma~\ref{lem-meas-est} hold for $\Bhat(x,r)$ as well,
for example in the following convenient form with $0 <c_1 \le c_2$:
\begin{equation} \label{eq-dhaq}
  \muhq(\Bhat(x,r))\simeq \begin{cases}
     \dha(x)^q \mu\biggl( B \biggl(x,\displaystyle\frac{r}{\dha(x)^2}\biggr)\biggr),  
            &\text{if } r\le c_2\dha(x), \\[4mm]    
     r^q \mu(B(a,1/r)),
            &\text{if } c_1\dha(x)\le r\le 2, 
                    \end{cases} 
\end{equation}
where 
the comparison constants 
depend also on $c_1$ and $c_2$.
\end{remark}

\begin{cor}  \label{cor-muha-Ahlfors}
If $q=2Q$ and $\mu$ is Ahlfors $Q$-regular, 
then so is $\muhq$.

Here the comparison constants in the Ahlfors $Q$-regularity
for $\muhq$ only depend on $Q$ and 
the Ahlfors $Q$-regularity constants for $\mu$.
\end{cor}

\begin{proof}
Since $\mu$ is Ahlfors regular, 
$X$ must be uniformly perfect with constant $\ka$ only depending
on the Ahlfors regularity constants and $Q$, cf.\ 
Heinonen~\cite[Exercise~13.1]{Heinonen}.
Also $\mu$ is doubling
and \eqref{eq-dim-cond-at-a} holds with $s=Q$.
It thus follows from Corollary~\ref{cor-muha-doubl} that
$\muhq$ is doubling, with 
$C_\muhq$ only depending on 
$Q=q-s$ and the Ahlfors $Q$-regularity constants  for $\mu$.
The same is true for the comparison constants below.

If $\tfrac13 \dhat(x) \le r \le 2$, then
by Remark~\ref{rmk-Bhat} (with $c_1=c_2=\frac13$)
and the $Q$-regularity of~$\mu$,
\begin{equation*} 
   \muhq(\Bhat(x,r))
  \osr{eq-dhaq}{\simeq} r^q \mu(B(a,1/r)) 
   \simeq r^q \frac{1}{r^Q}
   = r^Q.
\end{equation*}
On the other hand, if 
$r <\tfrac13 \dhat(x)$ then
by Remark~\ref{rmk-Bhat} again
and the $Q$-regularity of~$\mu$,
\[  \muhq(\Bhat(x,r))
   \osr{eq-dhaq}{\simeq} \dhat(x)^q  
     \mu\biggl( B\biggl(x,\frac{r}{\dhat(x)^2} \biggr) \biggr) 
\simeq \dhat(x)^q  \biggl(\frac{r}{\dhat(x)^2} \biggr)^Q = r^Q.
   \qedhere
\]  
\end{proof}

One may ask if the condition $q>s$ is necessary
in Corollary~\ref{cor-muha-doubl}.
In fact it is, as we now show.

\begin{thm} \label{thm-doubling-necessary}
Assume that $X$ is uniformly perfect at the base point $a$ for radii $\ge1$
with constant $\ka$.
Also assume that $\mu$ and $\muhq$, with parameter $q>0$, are doubling 
measures.
Then there are $0<s<q$ and $C_s>0$ such that
\begin{equation*}  
\frac{\mu(B(a,r))}{\mu(B(a,R))} \ge C_s\Bigl( \frac{r}{R} \Bigr)^s
\quad \text{whenever } 1\le r\le R<\infty.
\end{equation*}
\end{thm}

\begin{proof}
Let $r \ge 1$.
By the uniform perfectness, there is $x_r$ such that
\begin{equation}   \label{eq-choose-xr}
     2r \le    |x_r| < 2\ka r.
\end{equation}
Hence 
$B(x_r,r) \subset A(r):=B(a,3\ka r) \setm \itoverline{B(a,r)}$.
It thus follows from 
the doubling property of $\mu$ and
Lemma~\ref{lem-comp-close-balls} that
\begin{equation*} 
  \mu(A(r))\le
\mu(B(a,3\ka r)) \simeq \mu(B(a,r))
\simeq \mu(B(x_r,r)
\le \mu(A(r))
\end{equation*}
and hence
\begin{equation} \label{eq-A(r)}
      \frac{\mu(B(a,r))}{r^q} 
\osr{def-muha}{\simeq} \muhq(A(r)).
\end{equation}
Next, we show that 
\begin{equation}   \label{eq-Ahat}
\Bhat(x_r,\tfrac14 \dhat(x_r)) \subset A(r) 
\subset \Bhat\biggl( \binfty, \frac{1}{1+r} \biggr) 
\setm \Bhat\biggl( \binfty, \frac{1}{3\ka(1+r)} \biggr).
\end{equation}
Indeed, the latter inclusion follows immediately from \eqref{eq-Bhat-infty}
while for $y \in \Bhat(x_r,\tfrac14 \dhat(x_r))$ we have
$\dhat(y)> \tfrac34 \dhat(x_r) \osr{eq-dha(x)}{=} 3/4(1+|x_r|)$ and hence
\[
   |y| \osr{eq-dha(x)}{=} \frac{1}{\dhat(y)}-1 
< \frac43 (1+|x_r|)-1 
  = \frac13 + \frac43 |x_r| 
\osr{eq-choose-xr}{<}3\ka r.
\]
Similarly, $\dhat(y)< \tfrac54 \dhat(x_r)\osr{eq-dha(x)}{=} 5/4(1+|x_r|)$ and hence
\[
   |y| \osr{eq-dha(x)}{=}
\frac{1}{\dhat(y)}-1 > 
\frac45 (1+|x_r|)-1 
  =  \frac45 |x_r|  - \frac15
\osr{eq-choose-xr}{>} r.
\]
This proves \eqref{eq-Ahat}.

Since $\muhq$ is doubling, \eqref{eq-Ahat} and Lemma~\ref{lem-comp-close-balls}  with 
$\rho:=1/(1+r) \osr{eq-choose-xr}{\simeq} \dhat(x_r)$
show that
\begin{align}   
\muhq(A(r)) 
&\osr{eq-Ahat}{\le} \muhq(\Bhat(\binfty,\rho))
\simeq \muhq(\Bhat(\binfty,\rho/3\ka)) \nonumber \\
&\simeq \muhq(\Bhat(x_r,\tfrac14 \dhat(x_r)))
\osr{eq-Ahat}{\le} \muhq(A(r)),
\label{eq-Bhat(rho)}
\end{align}
which together with \eqref{eq-Ahat} implies that there is $\theta>0$, independent of $\rho$,
such that
\begin{align*}
\muhq(\Bhat(\binfty,\rho)) 
&\osr{eq-Ahat}{\ge} \muhq(\Bhat(\binfty,\rho/3\ka)) 
    + \muhq(\Bhat(x_r,\tfrac14 \dhat(x_r)))\\
&\osr{eq-Bhat(rho)}{\ge} (1+\theta) \muhq(\Bhat(\binfty,\rho/3\ka)).
\end{align*}
Since this holds for all 
$r\ge 1$, i.e.\ all $0 < \rho \le \tfrac12$,
and $\muhq$ is doubling,
we easily obtain by iteration  that for some $0<\de<q$,
\begin{equation}   \label{eq-iter-rev-doubl}
\frac{\muhq(\Bhat(\binfty,\rho'))}{\muhq(\Bhat(\binfty,\rho''))} 
\simle \biggl( \frac{\rho'}{\rho''} \biggr)^\de
\quad \text{when } 0 < \rho' \le \rho''\le 1.
\end{equation}

From  \eqref{eq-A(r)}, \eqref{eq-Bhat(rho)}
and the doubling property of $\muhq$ we also obtain that
\begin{equation} \label{eq-Br-2}
\frac{\mu(B(a,r))}{r^q}  
\osr{eq-A(r)}{\simeq}  \muhq(A(r)) 
\osr{eq-Bhat(rho)}{\simeq} \muhq(\Bhat(\binfty,1/(1+r)))
 \simeq \muhq(\Bhat(\binfty,1/r)),
\end{equation}
since 
$r\ge 1$ 
and thus $1/(1+r)\simeq 1/r$.
Applying these estimates also with $R\ge r$ in place of $r$,
together with \eqref{eq-iter-rev-doubl}, yields
\[
\frac{\mu(B(a,r))}{\mu(B(a,R))} 
\osr{eq-Br-2}{\simeq} \Bigl( \frac{r}{R} \Bigr)^q
\frac{\muhq(\Bhat(\binfty,1/r))}{\muhq(\Bhat(\binfty,1/R))}
\osr{eq-iter-rev-doubl}{\simge} \Bigl( \frac{r}{R} \Bigr)^{q-\de}
\quad \text{for }R \ge r \ge  1.
\qedhere
\]
\end{proof}

\begin{proof}[Proof of Theorem~\ref{thm-doubling-intro}]
One direction follows from 
Corollary~\ref{cor-muha-doubl},
while the other one follows from
Theorem~\ref{thm-doubling-necessary}.
\end{proof}

\section{Preservation of the \texorpdfstring{\p}{p}-Poincar\'e inequality}
\label{sect-PI}

The following important result guarantees that $X$ 
with a sufficiently good Poincar\'e inequality is 
\emph{annularly  quasiconvex}, 
i.e.\ 
there is a constant $A\ge 1$ such that whenever $B\subset X$ 
is a ball and $y,z\in B\setm \tfrac12 B$, there is a rectifiable curve 
$\gamma \subset A B\setm(2A)^{-1}B$, 
connecting $y$ to $z$, with length $l_\gamma\le A d(y,z)$.

\begin{thm} \label{thm-Korte}
\textup{(Korte~\cite[Theorem~3.3]{korte07})}
Let $X$ be a complete metric space equipped with a doubling measure $\mu$
for which 
there are constants $C,\s>1$ such that
\begin{equation*}    
\frac{\mu(B(x,r))}{\mu(B(x,R))} \le C\Bigl( \frac{r}{R} \Bigr)^{\s}
\quad \text{whenever } x \in X \text{ and }0 < r\le R<\infty.
\end{equation*}
If $X$ supports a $\s$-Poincar\'e inequality,
then $X$ is annularly quasiconvex.
\end{thm}

We shall use the following similar condition near $\binfty$.

\begin{deff} \label{deff-ann-conn} 
We say that $X$ is \emph{annularly connected for large radii around $a$}
if there are constants $A,R_A\ge1$ such that 
any two points $x$ and $y$ with $|x|=|y|\ge R_A$
can be connected by a (not necessarily rectifiable) curve 
$\ga \subset B(a,A|x|) \setm \itoverline{B(a,|x|/A)}$.
\end{deff}

We do not need the curve $\ga$ to be rectifiable,
although under the additional 
assumptions in Theorem~\ref{thm-main-PI} 
below it follows quite easily from the quasiconvexity of $X$
(provided by Theorem~\ref{thm-PI-imp-quasiconvex})
that one can equivalently require
$\ga$ to be rectifiable and even of length 
$\simeq d(x,y)$.
Thus, in quasiconvex spaces, this condition is equivalent to 
``annular quasiconvexity for large radii around $a$'',
but the current formulation is more convenient for us.
At the same time, the condition is stronger than 
``sequentially annularly chainable around $a$'',
which only requires chainability for a sequence $r_j\to\infty$ of radii, cf.\ 
Bj\"orn--Bj\"orn--Shan\-mu\-ga\-lin\-gam~\cite[Definition~5.2]{BBSliouville}.

\begin{thm}\label{thm-main-PI}
Assume that $(X,d,\mu)$ is complete, unbounded  and
supports a \p-Poincar\'e  inequality,
with dilation constant $\la$,
and that the measure $\mu$ is doubling and satisfies~\eqref{eq-dim-cond-at-a}
for  some $s>0$.
Assume in addition that
$X$ is annularly connected for large radii around $a$,
with constants $A,R_A\ge1$.

Then the sphericalization $(\Xhat,\dha,\muhq)$, with $q>s$, also supports a 
\p-Poincar\'e inequality, with constants depending quantitatively 
on all of the constants above.
Moreover, the dilation constant $\lahat$
in the \p-Poincar\'e inequality on $\Xhat$
can be chosen as $\lahat=100\la \max\{A,2R_A\}$.
\end{thm}

Note that if $(\Xhat,\dha)$ is $\Lhat$-quasiconvex for some $\Lhat\ge1$,
then the dilation constant can be chosen to be $\Lhat$ instead,
by Corollary~4.40 in~\cite{BBbook}.
We stress the fact that the assumption~\eqref{eq-dim-cond-at-a}
and the condition $q>s$ is only needed to 
guarantee that $\muhq$ is doubling.
We will not use any of the more
precise estimates from Section~\ref{sect-doubling} in the
proof below.

The main idea of the proof, based on chains of subWhitney balls, 
is relatively standard and has been used 
for different measures in e.g.\ Haj\l asz--Koskela~\cite{HaKo}, 
Heinonen--Koskela~\cite{HeKo98}
and Li--Shan\-mu\-ga\-lin\-gam~\cite{LiShan}.

The main difference compared with the above papers is in Case~2
when treating the balls centred at $\binfty$.
Rather than estimating the measure of the set $\{|u|>t\}$,
we control the measure of the  Whitney balls $\Bhat$ in 
$\Xhat \setm \{\binfty\}$, for which
$\vint_{\Bhat} |u|\,d\muhq \simeq 2^k$.
Thus we do not need to use Lebesgue points of $u$ and our chains of Whitney
balls are finite. 
For this purpose we also construct a suitable
Whitney covering tailored to the ball $\Bhat(\binfty,r)$.
A minor difference is that our proof does not depend on the Carath\'eodory type
level set lemma, such as
\cite[Lemma~4.22]{Heinonen} or \cite[Lemma~3.5]{LiShan}.

\begin{proof}[Proof of Theorem~\ref{thm-main-PI}]
Note first that $\muhq$ is a doubling measure, by
Corollary~\ref{cor-muha-doubl},
since $X$ is connected (and thus uniformly perfect)
due to the Poincar\'e inequality.

We split the discussion into three cases:
for relatively small (subWhitney) balls 
away from infinity, 
for balls centred at infinity and for intermediate balls. 
Let $u:\Xhat\to\eR$ be an integrable
function 
with an upper gradient $\gh$ (with
respect to~$\dha$).
Then, by Lemma~\ref{lem-upper-grad},
\begin{equation}    \label{eq-g-gh}
g(y) :=  \dhat(y)^2\gh(y)
\end{equation}
is an upper gradient of $u$ with respect to $d$ in $X$.

{\bf Case 1:} 
$24\la r\le \dha(x)$.
By~\eqref{eq-comp-balls-da-dhat} and Lemma~\ref{lem-ball-est} we have 
\begin{equation}   \label{eq-comp-Bhat-B'}
\Bhat(x,r) 
\osr{eq-comp-balls-da-dhat}{\subset} B_a(x,4r) 
\osr{eq-incl-small-r}{\subset} B\biggl(x,\frac{6r}{\dhat(x)^2}\biggr) =:B' 
\quad \text{and}  \quad 
\la B' \os{\eqref{eq-incl-small-r},\eqref{eq-comp-balls-da-dhat}}{\subset} \Bhat(x,8\la r).
\end{equation}
Hence, as $\muhq$ is doubling,
\begin{equation}    \label{eq-est-muha-B'}
\muhq(\Bhat(x,8\la r)) \simeq 
\muhq(\Bhat(x,r))
\simeq \muhq(B') 
\simeq \muhq(\la B').
\end{equation}
Moreover, for all $y\in\la B'$,
\begin{equation*} 
\tfrac23 \dhat(x) 
\le \dhat(x) - 8 \la r
< \dhat(y)
< \dhat(x) + 8 \la r
\le \tfrac43 \dhat(x), 
\end{equation*}
and thus
\begin{equation} \label{eq-muha-mu}
d\muhq(y) \simeq \dhat(x)^q \,d\mu(y)
\quad \text{in }\la B'.
\end{equation}
The \p-Poincar\'e inequality on $X$, together with
\eqref{eq-est-muha-B'}, \eqref{eq-muha-mu}
and the doubling property of $\mu$, then yields
\begin{align}
\vint_{\Bhat(x,r)}|u-u_{B'}| \,d\muhq
& \simle \frac{1}{\dhat(x)^q\mu(B')} 
        \int_{B'} |u-u_{B'}| \dhat(x)^q\, d\mu \nonumber \\
&
= \vint_{B'} |u-u_{B'}|\, d\mu
\simle \frac{r}{\dhat(x)^2} 
      \biggl(\vint_{\lambda B'}g^p \, d\mu\biggr)^{1/p},
\label{eq-PI-in-Bl}
\end{align}
where $u_{B'}=\vint_{B'}u\,d\mu$.
Using \eqref{eq-muha-mu}
again, together with~\eqref{eq-g-gh}, 
the last integral becomes 
\[  
\biggl(\vint_{\lambda B'}g^p\,d\mu\biggr)^{1/p}
\simeq \biggl(\frac{\dhat(x)^q}{\muhq(\la B')} \int_{\la B'} \dhat(x)^{2p}\gh^p 
            \, \frac{d\muhq}{\dhat(x)^{q}} \biggr)^{1/p} 
=  \dhat(x)^{2} \biggl(\vint_{\la B'} \gh^p 
            \, d\muhq \biggr)^{1/p}.
\]  
Inserting this into~\eqref{eq-PI-in-Bl}, together with 
\eqref{eq-comp-Bhat-B'} and \eqref{eq-est-muha-B'},
yields
\[
\vint_{\Bhat(x,r)}|u-u_{B'}| \,d\muhq
\simle r \biggl(\vint_{\Bhat(x,8\la r)} \gh^p  \, d\muhq \biggr)^{1/p}.
\]
A standard argument 
using the triangle inequality
allows us to replace $u_{B'}$ in the left-hand
side by the integral average $\vint_{\Bhat(x,r)}u\,d\muhq$
at the cost of an extra factor~2 in the right-hand side.
This concludes the proof of Case 1.
Note that the dilation constant here is $8\la$.

{\bf Case 2:}
$x=\binfty$ and $r\le 1/8R_A$, where $R_A\ge 1$ 
is the constant from the annular connectedness of $X$.

{\bf Step 1:} (Whitney covering)
First, we construct a suitable Whitney type covering of $X=\Xhat\setm\{\binfty\}$
tailored for the ball $\Bhat(\binfty,r)$.
For each $z\in X$ choose the unique 
integer $l_z$
such that
\begin{equation}   \label{eq-def-l-for-z}
   2^{-l_z}r\le \dhat(z)<2^{1-l_z}r \quad \text{and let $r_z=c_0 2^{-l_z}r$,}
\end{equation}
where $c_0=1/120\la$.
The balls $\Bhat(z,\tfrac15 r_z)$, with $z\in X$,
cover $X$. 
Since 
$r_z\le c_0\dhat(z) \le c_0$, 
the balls in
$\mathcal B$ have uniformly bounded radii.
Hence the $5$-covering lemma
(see e.g.\ Heinonen~\cite[Theorem~1.2]{Heinonen})
provides us with a countable collection
$\B$ of balls of the type $\Bhat(z,r_z)$ so that 
\[
X = \bigcup_{\Bhat\in\B} \Bhat
\]
and the balls $\{\tfrac15\Bhat\}_{\Bhat\in \B}$ are pairwise disjoint.
The doubling property of $\muhq$, together with Lemma~\ref{lem-comp-close-balls}, 
then implies that for each $l$,
there are at most $M$ balls $\Bhat\in \B$ with radius $r_l= c_0 2^{-l}r$,
and the bound $M$ is independent of $r$ and $l$.

Note that if $\Bhat(z,r_z), \Bhat(z',r_{z'})\in\B$
are such that 
$\Bhat(z,r_z) \cap \Bhat(z',r_{z'}) \ne\emptyset$ 
and $r_{z'}\le r_z$, then $\dha(z,z') < 2r_z$, and hence,
using~\eqref{eq-def-l-for-z} and that $1/c_0=120\la\ge120$,
\begin{equation}   \label{eq-comp-near-Bhat-balls}
\frac{2r_{z'}}{c_0} \osr{eq-def-l-for-z}{>} \dhat(z') \ge \dhat(z) - \dhat(z,z') 
\osr{eq-def-l-for-z}{>} \frac{r_z}{c_0} - 2r_z > \frac{2r_z}{3c_0}.
\end{equation}
This
implies that $r_z<3r_{z'}$ and thus $r_z\le2r_{z'}$ by the 
dyadic construction of $r_z$ in \eqref{eq-def-l-for-z}.
Thus the balls in $\B$ have bounded overlap (at most $3M$).
Moreover, by \eqref{eq-def-l-for-z} again,
\begin{equation} \label{eq-z}
\Bhat(z,r_z) \osr{eq-def-l-for-z}{\subset} 
\Bhat(\binfty,2^{2-l_z}r) \setm  \Bhat(\binfty,2^{-l_z-1}r).
\end{equation}
In particular,
\begin{equation*} 
l_z\ge0 \quad \text{whenever }
 \Bhat(z,r_z) \in \Bcalprime:=\{\Bhat \in \B : \Bhat \cap \Bhat(\binfty,r) \ne \emptyset\}.
\end{equation*}
Let $l_0=\min\{l_z : \Bhat(z,r_z) \in \Bcalprime\}$.
Then $l_0 \in \{0,1\}$,
by \eqref{eq-z} and the connectedness of $X$.
Fix a ball $\Bhat_0=\Bhat(z_0,r_{z_0}) \in\Bcalprime$ so that
$|z_0| \le |z|$ for all balls $\Bhat=\Bhat(z,r_{z})\in\Bcalprime$.
($\Bhat_0$ exists as there are only finitely many balls with $l=l_0$.)

{\bf Step 2:} (Chain of balls)
Now let  a ball $\Bhat(z)=\Bhat(z,r_z) \in \Bcalprime$ be arbitrary but fixed.
We shall construct a chain of Whitney balls $\Bhat\in\B'$ from 
$\Bhat_0$ to $\Bhat(z)$.
Since $X$ is quasiconvex, by Theorem~\ref{thm-PI-imp-quasiconvex}, 
there is a rectifiable curve $\gat$  from $z_0$ to $z$ with length 
$l_{\gat} \le Ld(z_0,z) \le 2L|z|$ (with respect to $d$),
where $L$ only depends on the constants in the doubling condition
and the \p-Poincar\'e inequality for $\mu$.
Let $z'$ be the last point in $\gat$ such that $|z'|\le |z_0|$.
Then $|z'|= |z_0|$ (since $|z|\ge |z_0|$).
Write $\gat$ as $\gat=\gat_1+\gat_2$, where $\gat_1$ is the part of
$\gat$ from $z_0$ to $z'$ and $\gat_2$ is the part from $z'$ to~$z$.

Since $l_{z_0}=l_0\ge  0$, we have 
\begin{equation} \label{eq-5b}
\dhat(z_0) < 2r \le 
\frac{1}{4R_A}, 
\end{equation}
and therefore
\[
     |z_0| \osr{eq-dha(x)}{=} \frac{1}{\dhat(z_0)}-1 > 4R_A -1 > R_A,
\]
where $R_A\ge 1$ is the constant from the annular connectedness of $X$.
Hence there is a curve 
$\ga_1  \subset B(a,A|z_0|) \setm \itoverline{B(a,|z_0|/A)}$
from $z_0$ to $z'$.
Let $\ga_z=\ga_1+\gat_2$.

Along $\ga_z$, starting from $\Bhat_0$, 
we form a chain of balls 
$\{\Bhat_j\}_{j=0}^N\subset \B$
to $\Bhat(z)$
so that each ball $\Bhat\in\B$ appears in the chain at most once.
More precisely, for $j=0,1,\dots$\,, let $\Bhat_{j+1}$ 
be the first ball $\Bhat\in \B$
intersecting $\ga_z$, after $\ga_z$ has left $\Bhat_j$ for the last time.
(Pick any such ball $\Bhat$ if it is not unique.)
If needed, add $\Bhat(z)$ to the chain as the last ball 
so that we always have  $\Bhat_N=\Bhat(z)$.
(This chain need \emph{not} cover $\ga_z$.
What is important here is that consecutive balls intersect.)

\emph{Claim}\/: For every ball $\Bhat=\Bhat(x,r_x)$ in the chain we have
\begin{equation}   \label{eq-ball-inclusions}
40\la\Bhat \subset \Bhat(\binfty,3Ar)
\quad \text{and}  \quad 
\Bhat(z)\subset \La \Bhat,
\end{equation}
where $\La\ge 40\la$ is a constant independent of $z$, $r$ and $\Bhat$.
Indeed, by construction, 
there is $y \in \ga_z \cap \Bhat$.
As $y \in \ga_z$ we have
\begin{equation} \label{eq-5e}
\frac{|z_0|}{A} \le |y|  \le  \max\{A|z_0|,(2L+1)|z|\} \le A'|z|,
\end{equation}
where $A'=\max\{A,2L+1\}$.
Hence also
\begin{equation} \label{eq-5d<}
\dhat(y) 
\osr{eq-dha(x)}{=} \frac{1}{1+|y|} 
\osr{eq-5e}{\le} \frac{1}{1+|z_0|/A} \osr{eq-dha(x)}{\le} A\dhat(z_0)
\end{equation}
and
\begin{equation} \label{eq-5d>}
\dhat(y) 
\osr{eq-dha(x)}{=} \frac{1}{1+|y|}
\osr{eq-5e}{\ge} \frac{1}{1+A'|z|} \osr{eq-dha(x)}{\ge} \frac{\dhat(z)}{A'}.
\end{equation}
Since $y \in \ga_z \cap \Bhat$, we thus get that
\begin{equation} \label{eq-5c}
(1-c_0) \dhat(x) \osr{eq-def-l-for-z}{\le} \dhat(x) - r_x \le \dhat(y) 
\le \dhat(x) + r_x \osr{eq-def-l-for-z}{\le} (1+c_0) \dhat(x)
\end{equation}
and hence
\begin{equation}   \label{eq-dhat-x-z}
\frac{\dhat(z)}{(1+c_0)A'} 
\osr{eq-5d>}{\le} \frac{\dhat(y)}{1+c_0}
\osr{eq-5c}{\le} \dhat(x) 
\osr{eq-5c}{\le} \frac{\dhat(y)}{1-c_0}
\osr{eq-5d<}{\le} \frac{A\dhat(z_0)}{1-c_0}
\osr{eq-5b}{<} \frac{2Ar}{1-c_0}.
\end{equation}
By the choice of $l_z$ and $l_x$ in \eqref{eq-def-l-for-z}, this implies that
$l_x \ge \lmin$, 
where $\lmin<0$ is the smallest integer
such that $2^{-\lmin} < 2A/(1-\nobreak c_0)$.
Moreover, $r_z\simle r_x$,
\begin{equation*}     
\dhat(x)+40\la r_x \osr{eq-def-l-for-z}{\le} \tfrac43 \dhat(x) 
\osr{eq-dhat-x-z}{<} 3Ar
\quad \text{and}  \quad 
\dha(z,x) \le \dha(z)+\dha(x)
\osr{eq-dhat-x-z}{\simle} \dha(x) \simeq r_x,
\end{equation*}
which implies \eqref{eq-ball-inclusions} and proves the claim.

{\bf Step 3:} (Telescoping argument for $\Bhat(z)$)
For each $j=0,1,\dots,N$, let $r_j$ be the radius of $\Bhat_j$
in the chain from $\Bhat_0$ to $\Bhat_N=\Bhat(z)$.
Note that by 
\eqref{eq-comp-near-Bhat-balls}, 
we have $\tfrac12 r_j \le r_{j-1} \le 2r_j$ and hence
$\Bhat_{j-1} \subset 5\Bhat_j$ when $j \ge 1$.
Without loss of generality, we can assume that 
\[
u_{\Bhat_0}:= \vint_{\Bhat_0} u\,d\muhq =0.
\]
A standard telescoping argument and the triangle inequality show that
\begin{align}
\biggl| \vint_{\Bhat(z)} |u|\,d\muhq - \vint_{\Bhat_0} |u|\,d\muhq \biggr| 
&\le \sum_{j=1}^N \biggl| \vint_{\Bhat_j} |u|\,d\muhq - \vint_{\Bhat_{j-1}} |u|\,d\muhq \biggr| 
\nonumber\\
& \kern -6em \le \sum_{j=1}^N  \biggl( 
\vint_{\Bhat_j} \bigl| |u|- |u|_{5\Bhat_{j}} \bigr| \,d\muhq
  + \vint_{\Bhat_{j-1}} \bigl| |u| - |u|_{5\Bhat_j}  \bigr|\,d\muhq  \biggr),
\label{eq-5a}
\end{align}
where $|u|_{5\Bhat_j} = \vint_{5\Bhat_j} |u|\,d\muhq$.
The balls $5\Bhat_j$, $j=0,1,\dots,N$, can be treated by Case~1.
Note that $\gh$ is also an upper gradient for $|u|$.
We therefore get from the already proved \p-Poincar\'e inequality for 
$5\Bhat_j \supset \Bhat_{j-1}\cup \Bhat_j$,
with dilation $8\la$, and the doubling property of $\muhq$, 
that
\begin{align}  \label{eq-first-telescope}  
\biggl| \vint_{\Bhat(z)} |u|\,d\muhq - \vint_{\Bhat_0} |u|\,d\muhq \biggr|
& \osr{eq-5a}{\simle} \sum_{j=0}^N
\vint_{5\Bhat_j} \bigl| |u| - |u|_{5\Bhat_{j}} \bigr| \,d\muhq \nonumber \\
&\simle \sum_{j=0}^N r_j \biggl(  \vint_{40\la \Bhat_j} \gh^p \,d\muhq \biggr)^{1/p}.
\end{align} 
Since we have assumed that $u_{\Bhat_0}=0$, we also have that
\[
\vint_{\Bhat_0} |u| \,d\muhq 
\simle r_0 \biggl(  \vint_{8\la\Bhat_0} \gh^p \,d\muhq \biggr)^{1/p}
\simle r_0 \biggl(  \vint_{40\la\Bhat_0} \gh^p \,d\muhq \biggr)^{1/p}.
\]
Together with \eqref{eq-first-telescope}, this leads to
\begin{equation}   \label{eq-after-telescope}
\vint_{\Bhat(z)} |u|\,d\muhq 
\simle
\sum_{j=0}^N r_j \biggl(  \vint_{40\la \Bhat_j} \gh^p \,d\muhq \biggr)^{1/p}.
\end{equation}

Let $0<\eps<1$, to be chosen later. 
Then we have for the radii $r_j$  in the chain 
that
\begin{equation} \label{eq-Ceps} 
\sum_{j=0}^N \Bigl(\frac{r_j}{r}\Bigr)^\eps
  \le M  \sum_{i=\lmin}^\infty (c_0 2^{-i})^\eps =: C_\eps  <\infty,
\end{equation} 
where $M$ is the maximal number of balls with the same radius in the Whitney 
covering $\B$.
Then
\[
\frac{1}{C_\eps}  \sum_{j=0}^N  \Bigl(\frac{r_j}{r}\Bigr)^\eps \vint_{\Bhat(z)} |u|\,d\muhq  
\osr{eq-Ceps}{\le} \vint_{\Bhat(z)} |u|\,d\muhq 
\osr{eq-after-telescope}{\simle}
\sum_{j=0}^N r_j \biggl(  \vint_{40\la \Bhat_j} \gh^p \,d\muhq \biggr)^{1/p}.
\]
By comparing the sums, we conclude that among the balls $\Bhat_j$
in the chain, there exists 
$\Bhat^*(z):=\Bhat(z^*,r^*_z) \in \B$ such that 
\begin{equation}   \label{eq-choose-By}
\vint_{\Bhat(z)} |u|\,d\muhq  \simle r \Bigl(\frac{r^*_z}{r}\Bigr)^{1-\eps}
\biggl(\vint_{40\la \Bhat^*(z)}\gh^p \, d\muhq\biggr)^{1/p}.
\end{equation}
Note that $\Bhat=\Bhat^*(z)$ satisfies \eqref{eq-ball-inclusions}.
The doubling property of $\muhq$ therefore
implies that (see e.g.\ \cite[Lemma~3.3]{BBbook})
\begin{equation*} 
\Bigl(\frac{r_z^*}{r} \Bigr)^{\sh} \simle \frac{\muhq(\Bhat^*(z))}{\muhq(\Bhat_{0})},
\quad \text{where } \sh=\log_2 C_{\muhq}.
\end{equation*}
Inserting this into~\eqref{eq-choose-By}, 
we obtain 
\begin{equation}   \label{eq-est-average}     
\vint_{\Bhat(z)} |u|\,d\muhq   
\le  Cr \frac{\muhq(\Bhat^*(z))^{(1-\eps)/{\sh}-1/p} }{\muhq(\Bhat_{0})^{(1-\eps)/{\sh}} }
      \biggl(\int_{40\la \Bhat^*(z)}\gh^p \,d\muhq\biggr)^{1/p}.
\end{equation}
Note that $C$, as well as the comparison constants in the rest of the calculations,
depends on the fixed parameters and on $\eps$, but not on $u$, $z$, $z_0$ or $r$.

{\bf Step 4:} (Final estimates)
For  each  $k \in \Z$, let $\B_k'$ be the collection of balls 
$\Bhat(z)\in\Bcalprime$ for which  
\begin{equation}   \label{eq-choose-B_k}
2^{k} < \vint_{\Bhat(z)} |u|\,d\muhq \le 2^{k+1}.
\end{equation}
Note that for $G_k:=\bigcup_{\Bhat\in\B_k'} \Bhat$,
\begin{equation}   \label{eq-int_G_k}
\int_{G_k} |u|\,d\muhq \le \sum_{\Bhat\in\B_k'} 2^{k+1} \muhq(\Bhat)
   \simle 2^{k} \muhq(G_k),
\end{equation}
since each ball $\Bhat\in\B$ meets at most a bounded number of balls from $\B$ 
(as their radii differ from $\Bhat$ by at most a factor $2$).
We therefore need to estimate $\muhq(G_k)$.

To this end, for each $\Bhat= \Bhat(z)\in \B_k'$, let $\Bhat^*(z)$
be as in~\eqref{eq-est-average}  and 
recall that \eqref{eq-ball-inclusions} holds for $\Bhat^*(z)$.
Let $\La\ge 40\la$ be as in \eqref{eq-ball-inclusions}.
Then the balls 
$\{\La\Bhat^*(z)\}_{\Bhat(z)\in\B_k'}$ cover the set $G_k$   
and 
have uniformly bounded radii $\le c_0\La$.
Thus,
by the $5$-covering lemma (see e.g.\ Heinonen~\cite[Theorem~1.2]{Heinonen}), 
there is a subcollection $\D_k$
of $\B_k'$ such that the balls $\La\Bhat^*(z)$, with $\Bhat(z)\in\D_k$,
are pairwise disjoint and 
\begin{equation}   \label{eq-est-G_k}
\muhq(G_k) \le \sum_{\Bhat(z)\in\D_k} \muhq(5\La\Bhat^*(z)) 
\simle \sum_{\Bhat(z)\in\D_k} \muhq(\Bhat^*(z)),
\end{equation}
by the doubling property of $\muhq$.
Now, choose $0<\eps<1$ such that 
\[
\tau:=1-\frac{(1-\eps)p}{\sh} >0.
\]
Then we conclude from \eqref{eq-est-average}  and \eqref{eq-choose-B_k}
 that for $\Bhat(z)\in\B_k'$,
\begin{equation}     \label{poinc-est}
2^{kp} \muhq(\Bhat^*(z))^{\tau}\le \frac{(Cr)^p}{\muhq(\Bhat_{0})^{1-\tau}}
            \int_{40\la \Bhat^*(z)}\gh^p \,d\muhq.
\end{equation}
Since $l_{z_0}=l_0\in \{0,1\}$, we 
also get from \eqref{eq-def-l-for-z} and
Lemma~\ref{lem-comp-close-balls} that 
\begin{equation} \label{eq-B-3Ar}
    \muhq(\Bhat_{0}) 
   \simeq \muhq(\Bhat(\binfty,r))
   \simeq \muhq(\Bhat(\binfty,3Ar)).
\end{equation}
Inserting this and~\eqref{poinc-est} into~\eqref{eq-est-G_k} gives
\begin{align}
\muhq(G_k)
& \osr{eq-est-G_k}{\simle} \sum_{\Bhat(z)\in\D_k} \muhq(\Bhat^*(z)) \nonumber\\
&\osr{poinc-est}{\simle} \frac{r^{p/\tau}}{2^{kp/\tau}\muhq(\Bhat_{0})^{1/\tau-1}}
    \sum_{\Bhat(z)\in\D_k}  \biggl( \int_{40\la\Bhat^*(z)}\gh^p \,d\muhq \biggr)^{1/\tau} \nonumber\\
&
\osr{eq-B-3Ar}{\simle}
\frac{\muhq(\Bhat(\binfty,r))}{2^{kp/\tau}}   
    \biggl(r^{p}  \vint_{\Bhat(\binfty,3Ar)} \gh^p \,d\muhq \biggr)^{1/\tau},
    \label{eq-est-G_j}
\end{align}
where  in the last inequality we have also used that the balls 
\[
40\la\Bhat^*(z)\osr{eq-ball-inclusions}{\subset} \La\Bhat^*(z) \cap \Bhat(\binfty,3Ar)
\]
are pairwise disjoint (by the construction of $\D_k$)
and that  $\tau < 1$.
Finally, by considering 
\begin{equation}     \label{eq-t0}
2^k\le t_0 :=  
  \biggl(r^p \vint_{\Bhat(\binfty,3Ar)}\gh^p \,d\muhq\biggr)^{1/p}
\end{equation}
and $2^k> t_0$ separately, we obtain from~\eqref{eq-est-G_j}
and  the doubling property of $\muhq$ that
\begin{align}  
\vint_{\Bhat(\binfty,r)}  |u| \,d\muhq 
&\osr{eq-int_G_k}{\simle} 
\frac{1}{\muhq(\Bhat(\binfty,r))}
\sum_{k\in\Z} 2^k \muhq(G_k) \nonumber \\
&\osr{eq-est-G_j}{\simle} t_0  + \sum_{2^k>t_0} 
   \frac{2^k}   {2^{kp/\tau}} 
   \biggl(r^{p}\vint_{\Bhat(\binfty,3Ar)}\gh^p \,d\muhq\biggr)^{1/\tau}.
  \label{eq-PI-8Ar-wrong-LHS}
\end{align}
Since $p \ge 1 > \tau$, the series converges and hence
\begin{align*}
\vint_{\Bhat(\binfty,r)}  |u| \,d\muhq 
&\osr{eq-PI-8Ar-wrong-LHS}{\simle}
t_0 
+ t_0^{1-p/\tau}   \biggl( r^p \vint_{\Bhat(\binfty,3Ar)}\gh^p \,d\muhq\biggr)^{1/\tau}
\\
&\osr{eq-t0}{\simeq} \biggl(r^p \vint_{\Bhat(\binfty,3Ar)}\gh^p \,d\muhq\biggr)^{1/p}.
\end{align*}
A standard argument using the triangle inequality
allows us to replace $|u|$  in the left-hand side by  $|u-u_{\Bhat(\binfty,r)}|$.
This concludes the proof of Case 2.
Note that the dilation constant here is $3A$.

{\bf Case~3:} 
$\dha(x)<24\la r$ and $r\le 1/200\la R_A$. 
In this case we use the conclusion of Case~2 as an aid, since
$\Bhat(x,r)\subset \Bhat(\binfty, 25\la r)=: \Bhat_\binfty$.
We then have, using also 
the triangle inequality and
the doubling property of $\muhq$,  that
\[   
\vint_{\Bhat(x,r)}|u-u_{\Bhat(x,r)}|\, d\muhq
\le 2 \vint_{\Bhat(x,r)}|u-u_{\Bhat_\binfty}|\, d\muhq\\
    \simle \vint_{\Bhat_\binfty}|u-u_{\Bhat_\binfty}|\, d\muhq.
\]   
The last integral is estimated using the 
\p-Poincar\'e inequality
for the ball $\Bhat_\binfty$ (with dilation $3A$) from Case~2.
Note that $25\la r\le 1/8R_A$. 
The inclusion
\[
3A \Bhat_\binfty = \Bhat(\binfty, 75A\la r) \subset \Bhat(x,100A\la r)
\]
completes the proof of Case 3, and hence of the theorem for 
balls with radii $\le 1/200\la R_A$.
Note that the dilation constant here is $100A\la$.

{\bf Case~4:} $r >1/200 \la R_A$.
Since $\Xhat$ is bounded, Theorem~1.3 in
Bj\"orn--Bj\"orn~\cite{BBsemilocal} extends 
the \p-Poincar\'e inequality to all radii.
More precisely, using the $5$-covering lemma 
(see e.g.\ Heinonen~\cite[Theorem~1.2]{Heinonen}),
we can cover $\Xhat$ by balls
$\Bhat_j'=\Bhat(x_j,1/400\la R_A)$ such that the balls
$\tfrac15 \Bhat_j'$ are pairwise disjoint.
The number of balls $N$ depends only on $C_\muhq$, $\la$ and $R_A$
(since $\diam \Xhat=1$, by \eqref{eq-diam}).
As $\Xhat$ is connected we can order the balls so that 
$\Bhat_j' \cap \bigcup_{k=1}^{j-1} \Bhat_k' \ne \emptyset$
for $j=2,\dots,N$.
Hence there is $\theta>0$ such that
\[
\muhq(2\Bhat_j' \cap A_j) \ge \theta \muhq(2\Bhat_j'), \quad j=2,\dots,N,
\quad \text{where } A_j=\bigcup_{k=1}^{j-1} 2\Bhat_{k}'.
\]
For a given ball $\Bhat=\Bhat(x,r)$ with $r >  1/200\la R_A$
it then follows from the proof of 
Theorem~4.4 in~\cite{BBsemilocal} 
that Lemma~4.11 in~\cite{BBsemilocal} is only used at most $N$ times.
Hence, the constants
in the resulting \p-Poincar\'e inequality for $\Bhat$ only
depend on the fixed parameters,
with the dilation constant $\la'= 200\la R_A$
since then $\la' \Bhat = \Xhat$.
This concludes the proof of Case 4.
Note that the dilation constant $\lahat=100 \la \max\{A,2R_A\}$ can be used
for all balls in all the Cases 1--4.
\end{proof}

\section{Consequences for upper gradients and capacities}
\label{sect-conseq-ug}

\emph{In this section we assume that the assumptions in Theorem~\ref{thm-main-PI}
are satisfied.
It thus follows from  Corollary~\ref{cor-muha-doubl} and 
Theorem~\ref{thm-main-PI} that $\muhq$ is doubling and 
supports a \p-Poincar\'e inequality on $\Xhat$.
Recall that $(\Xhat,\dhat)$ is complete by construction.}

\medskip

We begin with some general results.
The following lemma replaces Lemma~3.2 in~\cite{BBLi}, which  
states that the two moduli are always equal when $q=2p$.
The equality fails when $q\ne2p$,
but here it is enough that the zero modulus families
coincide.

\begin{lem} \label{lem-Ga}
Let $\Ga$ be a family of compact rectifiable curves in $X$.
Then $\Ga$ has zero \p-modulus with respect to $(X,d,\mu)$
if and only if it has zero \p-modulus with respect to $(\Xhat,\dhat,\muhq)$.
\end{lem}

\begin{proof}
Let $\Ga_m:=\{\ga \in \Ga : \ga \subset B(a,m)\}$.
As $d$ and $\mu$ are comparable to $\dhat$ and $\muhq$ on $B(a,m)$,
we see that $\Ga_m$ has zero \p-modulus with respect to $(X,d,\mu)$
if and only if it has zero \p-modulus with respect to $(\Xhat,\dhat,\muhq)$.
Since this is true for all $m \in \Z$, the same is true for 
$\Ga=\bigcup_{m=1}^\infty \Ga_m$.
\end{proof}

The following relation 
extends Lemma~\ref{lem-upper-grad} to \p-weak upper gradients.
Note that $g$ is measurable with respect to $\mu$ if and only if
it is measurable with respect to $\muhq$.
The proof is similar to the proof of
Lemma~3.3 in~\cite{BBLi}, but relies
on Lemma~\ref{lem-Ga} above 
instead of Lemma~3.2 in~\cite{BBLi}.
Also the proofs of most of the other results
in this section are similar to proofs of corresponding 
results in~\cite{BBLi}.
For the reader's convenience we provide the short proofs.
Recall from the beginning of Section~\ref{sect-sphericalization}
that $|x|:=d(x,a)$.

\begin{cor} \label{cor-pwug}
Let $E \subset X$ be measurable,
$u:E \to \eR$ be a function, $g:E \to [0,\infty]$
be measurable, and
\[
    \gh(x)=g(x)(1+|x|)^2, 
    \quad  x \in E.
\]
Then $g$ is a \p-weak upper gradient of $u$ in $E$ with respect to
$(d,\mu)$ if and only if $\gh$ is a 
\p-weak upper gradient of $u$ in $E$ with respect to
$(\dha,\muhq)$.
\end{cor}

\begin{proof}
Assume that $g$ is a \p-weak upper gradient of $u$ with respect to
$(d,\mu)$.
Let $\Ga$ be the family of exceptional nonconstant rectifiable
curves in $E$ for which \eqref{ug-cond} fails.
Then $\Ga$ has zero \p-modulus with respect to $(X,d,\mu)$ (since $g$ is a \p-weak upper gradient)
and thus also with respect to $(\Xhat,\dhat,\muhq)$, by Lemma~\ref{lem-Ga}.
Let $\ga:[0,1] \to E$ be a nonconstant rectifiable curve not in $\Ga$.
Then, using \eqref{eq-dsha},
\[
    |u(\ga(0))-u(\ga(1))| \le \int_\ga g\, ds 
    = \int_\ga \gh\, d\sh. 
\]
Thus $\gh$ is
a \p-weak upper gradient of 
$u$ with respect to $(\dha,\muhq)$.
The converse implication is shown similarly.
\end{proof}

From now on, for a function $u$ we 
let $g_u$ and $\gh_u$ be the minimal \p-weak upper gradients with respect
 to $(X,d,\mu)$ and $(\Xhat,\dha,\muhq)$, respectively,
when they exist.
We also let $\Cp$ and $\Cphat$ denote the Sobolev capacities
with respect to $(X,d,\mu)$ and $(\Xhat,\dha,\muhq)$, respectively.

\begin{lem}   \label{lem-Sob-spc-equal}
Let $E \subset X$ be measurable. 
\begin{enumerate}
\item \label{a-Np}
If $E$ is bounded, then $\Np(E,d,\mu)= \Np(E,\dha,\muhq)$, 
with comparable norms, 
and $\Dp(E,d,\mu)= \Dp(E,\dha,\muhq)$, 
with comparable seminorms.
\item \label{a-Nploc}
$\Nploc(E,d,\mu)= \Nploc(E,\dha,\muhq)$
and $\Dploc(E,d,\mu)= \Dploc(E,\dha,\muhq)$.
\end{enumerate}
\end{lem}

\begin{proof}
By Corollary~\ref{cor-pwug}, $g_u \simeq \gh_u$ a.e.\ in $E$ if $E$ is bounded.
As also $d \simeq \dha$ and $\mu \simeq \muhq$ in $E$,
we immediately get \ref{a-Np}, while
\ref{a-Nploc} follows directly from \ref{a-Np}, 
since $g_u$ 
only depends on the function locally.
\end{proof}

\begin{cor} \label{cor-Dp-equiv}
Let $E \subset X$ be measurable and  $q=2p$.
If $u \in \Dploc(E,d,\mu)=\Dploc(E,\dha,\muhq)$,
then
\begin{equation}  \label{eq-int-gu-guhat}
\gh_u(x) = g_u(x)(1+|x|)^2 
\quad \text{and} \quad
     \int_E g_u^p \,d\mu = \int_E \gh_u^p \,d\muhq.
\end{equation}
In particular, $\Dp(E,d,\mu) = \Dp(E,\dha,\muhq)$.
\end{cor}

\begin{proof}
Note first that $\Dploc(E,d,\mu)=\Dploc(E,\dha,\muhq)$, by 
Lemma~\ref{lem-Sob-spc-equal}\ref{a-Nploc},
and $\gh_u(x) =  g_u(x)(1+|x|)^2$, by Corollary~\ref{cor-pwug}.
Therefore,
\[
   \int_E \gh_u^p \,d\muhq 
  \os{\eqref{def-muha},\eqref{eq-int-gu-guhat}}{=}
  \int_E g_u(x)^p (1+|x|)^{2p}  \frac{d\mu(x)}{(1+|x|)^{2p}}
    = \int_E g_u^p \, d\mu,
\]
from which it follows directly
that $u \in \Dp(E,\dha,\muhq)$ 
if and only if $u \in \Dp(E,d,\mu)$.
\end{proof}

\begin{lem} \label{lem-Cp=0-X}
Let $E \subset X$.
Then $\Cphat(E)=0$ if and only if $\Cp(E)=0$.

In particular, if $X$ is Ahlfors $Q$-regular and $p>Q$,
then $\Cphat(E)=0$ only if $E=\emptyset$.
\end{lem}

\begin{proof}
By Lemma~\ref{lem-Sob-spc-equal}\ref{a-Np},
the $\Np$-norms are comparable in $B(a,j)$, $j \ge 1$, 
from which it follows 
that $\Cpprime(E \cap B(a,j))=0$ if and only if
$\Cphatprime(E \cap B(a,j))=0$, where $\Cpprime$ and $\Cphatprime$
are capacities with $B(a,j)$ as the underlying space.
These capacities have the same zero sets in $B(a,j)$ as $\Cp$ and $\Cphat$,
respectively, by e.g.\  Lemma~2.24 in \cite{BBbook}.
The countable subadditivity of the capacities then 
shows that $\Cp(E)=0$ if and only if $\Cphat(E)=0$.

As for the second part, 
since $p>Q$,
Corollary~5.39 in \cite{BBbook} 
shows that nonempty sets have positive \p-capacity.
\end{proof}

The classification of unbounded metric spaces as parabolic and
hyperbolic plays an important role for the point $\binfty$. 
For this we first need to define the condenser capacity $\cp$.
We only need this for bounded sets $E$.
(The definition for unbounded $E$ is more complicated, see
 Bj\"orn--Bj\"orn~\cite{BBglobal}.)

\begin{deff} \label{deff-cp}
Let $\Om\subset X$ be a (possibly unbounded) open set.
The \emph{condenser capacity} of a bounded set
$E \subset \Om$ with respect to 
$\Om$ is
\begin{equation*} 
\cp(E,\Om) = \inf_u\int_{\Om} g_u^p\, d\mu,
\end{equation*}
where the infimum is taken over all 
$u \in \Np(X)$ such that $\chi_E \le u \le \chi_\Om$.
If no such function $u$ exists then $\cp(E,\Om)=\infty$.
\end{deff}

\begin{deff} \label{def-p-par}
Assume that $X$ is unbounded.
Then $X$ is called \emph{\p-hyperbolic}
if $\cp(K,X)>0$ for some compact set $K\subset X$.
Otherwise, $X$ is \emph{\p-parabolic}. 
\end{deff}

The following 
characterization of \p-parabolicity was obtained 
in Bj\"orn--Bj\"orn--Lehrb\"ack~\cite{BBLintgreen}.
(Theorem~1.0.1 in Keith--Zhong~\cite{KeithZhong}
provides us with the better Poincar\'e inequality required in~\cite{BBLintgreen}.)

\begin{thm}   \label{thm-p-parab}
\textup{(Theorem~5.5 in~\cite{BBLintgreen})}
Assume that $p>1$.
Then $X$ is \p-parabolic 
if and only if 
\begin{equation*}  
\int_{1}^\infty \biggl( \frac{r}{\mu(B(a,r))} \biggr)^{1/(p-1)} \,dr=\infty.
\end{equation*}
\end{thm}

Based on this and Proposition~5.3 in~\cite{BBLintgreen}
we obtain the following result.

\begin{lem}   \label{lem-Cp=0-infty}
Assume  that $p>1$.
Then $\Cphat(\{\binfty\}) = 0$ if and only if 
\[
\int_1^\infty \biggl( \frac{r^{q-2p+1}}{\mu(B(a,r))} \biggr)^{1/(p-1)}\, dr = \infty.
\]
In particular, if $q=2p$ then $\Cphat(\{\binfty\}) = 0$ if and only if $X$ 
is \p-parabolic.
\end{lem}

\begin{proof}
Proposition~5.3 in Bj\"orn--Bj\"orn--Lehrb\"ack~\cite{BBLintgreen} shows that
$\Cphat(\{\binfty\}) = 0$ if and only if 
\[
\int_0^1 \biggl( \frac{\rho}{\muhq(\Bhat(\binfty,\rho))} \biggr)^{1/(p-1)} 
\,d\rho= \infty.
\]
Remark~\ref{rmk-Bhat}
and the change of variable $r=1/\rho$ then prove the first statement.

The statement for $q=2p$ now
follows immediately from Theorem~\ref{thm-p-parab}.
\end{proof}

\section{\texorpdfstring{\p}{p}-harmonic functions 
under sphericalization}
\label{sect-pharm}

\emph{In this section we assume that the assumptions in Theorem~\ref{thm-main-PI}
are satisfied with $p>1$ and $q=2p$.
Thus,
by  Corollary~\ref{cor-muha-doubl} and 
Theorem~\ref{thm-main-PI},  $\muhtp$ is doubling and 
supports a \p-Poincar\'e inequality on $\Xhat$.
We also assume 
that $\Om \subset X$ is   a nonempty open set.}

\medskip

In the rest of the paper  
we apply sphericalization and the results from the earlier 
sections to obtain new results about \p-harmonic functions and the 
corresponding Dirichlet
problem on unbounded open sets,  
from similar results in bounded open sets.

Recall from Lemma~\ref{lem-meas-est} that we need 
\begin{equation*} 
 q > \uqi := \inf 
         \biggl\{s>0 : \frac{\mu(B(a,r))}{\mu(B(a,R))} \simge \Bigl( \frac{r}{R} \Bigr)^s
\text{ for } 1\le r\le R<\infty\biggr\}.
\end{equation*}
Hence 
our condition on $p$ is that $p> \max\{1,\tfrac12 \uqi\}$,
which in Ahlfors $Q$-regular spaces means that 
$p > \max \{1,\tfrac12 Q\}$.
It is quite easy to see that $\uqi$ is independent of the choice
of the base point $a \in X$.

Note that $\Rn$ equipped with a \p-admissible
measure $d\mu =w\,dx$, as in Heinonen--Kilpel\"ainen--Martio~\cite{HeKiMa},
is included here provided that $p> \max\{1,\tfrac12 \uqi\}$.

\begin{remark} \label{rmk-comp-BBLi} 
With Theorem~\ref{thm-main-PI},  we can treat more situations
than in Bj\"orn--Bj\"orn--Li~\cite{BBLi}.
The main advantage compared with \cite{BBLi} is that 
we do not require Ahlfors regularity.
But even for Ahlfors $Q$-regular spaces, we can treat a wider range of \p.
More precisely, we allow for all $p>Q$ and only require 
a \p-Poincar\'e inequality for $\mu$, 
while $(2-Q)p<Q$ and a stronger $t_p$-Poincar\'e inequality with
$t_p=pQ/(2p-Q)<Q$ were required  in~\cite{BBLi}. 
Also the case $Q=1$, which was 
omitted in~\cite{BBLi}, is included here.
The case $Q/2<p\le Q$, with $Q>1$, was however already covered by \cite{BBLi}.
Note that $Q \ge 1$, since $X$ is connected (due
to the Poincar\'e inequality) and unbounded,   
and $Q$ is the Hausdorff dimension of $X$,
see Heinonen~\cite[p.~62]{Heinonen}.

Even under the assumptions in~\cite{BBLi} our results
below are more general since they use the 
more refined capacity $\bCp$ and the larger space 
$\Dp(\clOm,d,\mu)$ (instead of $\Cp$ and $\Dp(X,d,\mu)$).
Results as in  Theorems~\ref{thm-trich-intro} and \ref{thm-barrier+local-intro}
were not considered in~\cite{BBLi}.
\end{remark}

At the same time, we need to explicitly assume in addition that $X$ 
is annularly connected for large radii around $a$, 
since for $p>Q$ this does not follow
from Theorem~\ref{thm-Korte}, unlike in \cite{BBLi} where annular quasiconvexity
followed from the better $t_p$-Poincar\'e inequality
(and Theorem~\ref{thm-Korte}).
In our setting, it is only guaranteed by Theorem~\ref{thm-Korte} 
if the measure $\mu$ grows sufficiently fast at $\binfty$
and satisfies a sufficiently good Poincar\'e inequality.

We start this section by some necessary preliminaries about
\p-harmonic functions and Perron solutions.

\begin{deff} \label{def-quasimin}
A function $u \in \Nploc(\Om,d,\mu)$ is a
\emph{minimizer}
in $\Om$
(with respect to $(d,\mu)$) if 
\[ 
      \int_{\phi \ne 0} g^p_u \, d\mu
           \le \int_{\phi \ne 0} g_{u+\phi}^p \, d\mu
           \quad \text{for all } \phi \in \Lipc(\Om),
\] 
where $g_u$ and $g_{u+\phi}$ are the minimal \p-weak upper gradients of $u$
and $u+\phi$ with respect to $(d,\mu)$.
Here, $\Lipc(\Om)$ denotes the space of Lipschitz functions
with compact support in $\Om$.

Minimizers
with respect to $(\dha,\muhtp)$ are defined analogously.
A \emph{\p-harmonic function} is a continuous minimizer.
\end{deff}

We are primarily interested in solvability and uniqueness for
the Dirichlet (boundary value) problem
for \p-harmonic functions in unbounded open sets, 
and the associated boundary regularity. 
The most general way of treating the Dirichlet problem
is to consider Perron solutions, 
for which we need superharmonic functions.

\begin{deff} \label{deff-superharm-class}
A function $u : \Om \to (-\infty,\infty]$ is 
\emph{superharmonic} in $\Om$ if
\begin{enumerate}
\renewcommand{\theenumi}{\textup{(\roman{enumi})}}%
\item \label{cond-a} $u$ is lower semicontinuous;
\item \label{cond-b} 
 $u$ is not identically $\infty$ in any component of $\Om$;
\item \label{cond-c}
for every nonempty open set $G \Subset \Om$ 
and every function $v\in C(\clG)$ which is \p-harmonic in $G$ and such
that $v\le u$ on $\bdy G$, we have $v\le u$ in $G$.
\end{enumerate}
\end{deff}

This definition of superharmonicity is the same as the one
usually used in the Euclidean literature, e.g.\ in
Hei\-no\-nen--Kil\-pe\-l\"ai\-nen--Martio~\cite[Section~7]{HeKiMa}.
It is equivalent to other definitions of superharmonicity on metric spaces,
by Theorem~6.1 in Bj\"orn~\cite{ABsuper}
(or \cite[Theorem~14.10]{BBbook}).
(Note that in~\ref{cond-c} in \cite{ABsuper} and~\cite{BBbook}
it is required that $\Cp(X \setm G)>0$,
but here that condition is redundant  since $G$ is bounded and $X$ is unbounded.
It plays a role only when $\Om=X$ and $X$ is bounded.)

Minimizers, \p-harmonic functions 
and superharmonic functions
were introduced in metric spaces by 
Shan\-mu\-ga\-lin\-gam~\cite{Sh-harm},
Kinnunen--Shan\-mu\-ga\-lin\-gam~\cite{KiSh01}
and Kinnunen--Martio~\cite{KiMa02}.
In~\cite{KiSh01} 
it was also shown that 
under the assumptions of doubling and a \p-Poincar\'e inequality,
minimizers 
can be modified on sets of zero capacity to become continuous 
and thus \p-harmonic functions.

Our choice of the sphericalization measure $\muhtp$ leads to the 
following invariance result which is crucial for the applications in this
section.
The proof in Bj\"orn--Bj\"orn--Li~\cite{BBLi} (based on 
\eqref{eq-int-gu-guhat}) applies verbatim here.

\begin{thm} \label{thm-harm-equiv-X-Xdot}
\textup{(Theorem~6.3 in \cite{BBLi})}   
A function  $u:\Om\to\R$ is a minimizer  in $\Om$ with respect to 
$(d,\mu)$
if and only if it is a minimizer in $\Om$ with respect to 
$(\dha,\muhtp)$.

Consequently, the notions of \p-harmonicity and superharmonicity 
are also the same with respect to $(d,\mu)$ and $(\dha,\muhtp)$.
\end{thm}

We are now ready to define the Perron solutions.
We consider the Dirichlet problem with respect to the 
\emph{extended boundary}
$\bdhat\Om$ corresponding to $\Xdot$, where
\[
\bdhat \Om:= \begin{cases}   
     \bdy\Om \cup \{\binfty\},   &  \text{if $\Om$ is unbounded,}  \\ 
     \bdy \Om,  & \text{otherwise.}  
\end{cases}
\]
This is in accordance with the definitions used in 
Heinonen--Kilpel\"ainen--Martio~\cite{HeKiMa},
Hansevi~\cite{Hansevi2} and Bj\"orn--Bj\"orn~\cite{BBglobal}.

\begin{deff}   \label{def-Perron}
Given 
$f : \bdhat \Om \to \eR$, let $\UU_f$ be the set of all 
superharmonic functions $u$ on $\Om$, bounded from below,  such that 
\begin{equation} \label{eq-def-Perron} 
	\liminf_{\Om \ni y \to x} u(y) \ge f(x) 
\end{equation} 
for all $x \in \bdhat \Om$.
The \emph{upper Perron solution} of $f$ is then defined to be
\[ 
    \uP f (x) = \inf_{u \in \UU_f}  u(x), \quad x \in \Om,
\]
while the \emph{lower Perron solution} of $f$ is defined by
\(
    \lP f = - \uP (-f).
\)
If $\uP f = \lP f$
and it is real-valued, then we let $Pf := \uP f$
and $f$ is said to be \emph{resolutive} with respect to $\Om$.
\end{deff}

The limit in \eqref{eq-def-Perron}  can be equivalently taken 
with respect to $\dha$ or $d$.
Thus, by Theorem~\ref{thm-harm-equiv-X-Xdot}, Perron
solutions with respect to $(d,\mu)$ and $(\dha,\muhtp)$ are the same.

The Perron method makes sense whenever 
the extended boundary $\bdhat \Om \ne \emptyset$,
which holds  here as $X$ is unbounded.
The case when $\Cp(X \setm \Om)=0$  is rather exceptional, especially
when $X$ is also \p-parabolic (since then $\bdhat\Om$ has zero $\Cphat$-capacity
and is invisible for Newtonian functions),
cf.\  Propositions~11.2 and~11.9 in Bj\"orn--Bj\"orn~\cite{BBglobal}.
We therefore make the following assumption.
\begin{equation} \label{eq-Om-cond} 
\begin{tabular}{c}
\emph{In addition to the assumptions from the beginning of this section,} \\
\emph{we assume in the rest of the paper that $\Om\subset X$ is unbounded,} \\
\emph{and that $X$ is \p-hyperbolic or that $\Cp(X \setm \Om)>0$.}
\end{tabular}
\end{equation}

In each component of $\Omega$, $\uP f$ is either \p-harmonic or 
identically $\pm\infty$, by Theorem~4.1 in Bj\"orn--Bj\"orn--Shan\-mu\-ga\-lin\-gam~\cite{BBS2}
(or \cite[Theorem~10.10]{BBbook}).
Moreover, $\lP f \le \uP f$ for all $f: \bdhat \Om \to \eR$, 
by Kinnunen--Martio~\cite[Theorem~7.2]{KiMa02}
(or \cite[Theorem~9.39]{BBbook}).

As $\Om$ is always bounded as a subset of the sphericalization $\Xdot$,
we can now, in a similar way as in Bj\"orn--Bj\"orn--Li~\cite{BBLi}, use 
results about \p-harmonic functions on bounded sets 
and  transfer them  to \p-harmonic functions 
and Perron solutions on 
$\Om\subset X$ even for unbounded $\Om$ (with 
the extended boundary $\bdhat\Om$).
We shall not repeat all those results here and refer the
reader to \cite{BBLi}.
All results obtained in \cite[Sections~6--8]{BBLi} can now be obtained
under our more general assumptions~\eqref{eq-Om-cond} instead of
\cite[(5.1) and (5.2)]{BBLi}.
Here, we only mention a few resolutivity and invariance results 
that are new for unbounded domains,
even in weighted $\R^n$ (and partly also unweighted~$\R^n$).

To be able to include larger negligible sets in our perturbation
results for Perron solutions,
we will use the following capacity, introduced in 
Bj\"orn--Bj\"orn--Shan\-mu\-ga\-lin\-gam~\cite[Definition~3.1]{BBSdir}.
Here and later,  the closure $\clOm \subset X$ is taken within~$X$.

\begin{deff} \label{deff-bCp}
Let $\Om\subset X$ be a (possibly unbounded) open set.
For $E \subset \clOm$  let
\[
     \bCp(E;\Om)= \inf_{u \in \A_E} \|u\|_{\Np(\Om)}^p,
\]
where $u \in \A_E$ if $u \in \Np(\Om)$ satisfies
both $u \ge 1$ on $E \cap \Om$ and 
\[
      \liminf_{\Om \ni y \to x} u(y) \ge 1
	\quad \text{for all } x \in E \cap \bdy \Om.
\]
\end{deff}

This capacity is countably subadditive and an outer capacity,
by Propositions~3.2 and~3.3 in~\cite{BBSdir},
which can be used because of Theorem~1.3 in 
Eriksson-Bique--Poggi-Corradini~\cite{EB-PC} (cf.\  Theorem~\ref{thm-quasicont}).

It is easy to see that $\bCp(E;\Om) \le \Cp(E)$,
but the $\bCp$ capacity sometimes
allows for larger zero sets,
see e.g.\ \cite[Examples~10.1--10.7]{BBSdir},
which is the main reason for introducing it.
Moreover, just as in Lemma~\ref{lem-Cp=0-X}, we see that
$\bCp(E;\Om,d,\mu)=0$ if and only if $\bCp(E;\Om,\dhat,\muhtp)=0$.
Since we will only be interested in when this capacity is $0$ we will
simply write $\bCp(E;\Om)=0$.

The following  resolutivity and invariance result
improves upon Proposition~11.7 in~Bj\"orn--Bj\"orn~\cite{BBglobal},
where a similar result was obtained under considerably
more restrictive conditions 
on $h$, in particular that $h$ is bounded.
In \p-parabolic spaces
this result is due to Hansevi~\cite[Theorem~7.8]{Hansevi2}.

\begin{thm}   \label{thm-resol-C}
Assume that \eqref{eq-Om-cond}  holds.
Let $f \in C(\bdhat \Om)$ and assume that $h:\bdhat \Om \to \eR$
vanishes  $\bCp$-q.e.\ on $\bdry\Om$,
i.e.\ 
\[
\bCp(\{x \in \bdy \Om : h(x) \ne 0\};\Om)=0.
\]
When $X$ is \p-hyperbolic,  assume also that  
$h(\binfty)=0$.

Then both $f$ and $f+h$ are resolutive and $\Hp f = \Hp (f+h)$.
\end{thm}

\begin{proof}
This follows by sphericalization and the above discussion
from the corresponding result for bounded $\Om$ 
obtained in   
Bj\"orn--Bj\"orn--Shan\-mu\-ga\-lin\-gam~\cite[Proposition~9.2]{BBSdir}.
Note that $\Cphat(\{\binfty\})=0$ if $X$ is \p-parabolic, 
by Lemma~\ref{lem-Cp=0-infty}.
\end{proof}

The following result shows that resolutivity and
invariance hold also for Sobolev type boundary data.
Compared with Theorem~6.7 in~\cite{BBLi}, 
it applies to more general measures and spaces, but it also
only requires that $f\in \Dp(\clOm)$,  
rather than $f\in \Dp(X)$.
In \p-parabolic spaces,  this 
complements Theorems~7.5 and~7.6 from Hansevi~\cite{Hansevi2}.

This improvement has been made possible by recent results on quasicontinuity
of Newtonian functions in general metric spaces
due to Eriksson-Bique--Poggi-Corradini~\cite{EB-PC},
see the discussion before Theorem~\ref{thm-quasicont}.

\begin{thm} \label{thm-Dp-res}
Assume that \eqref{eq-Om-cond} holds.
Let $f:\clOm \cup \{\binfty\} \to \eR$
be such that $f\in\Dp(\clOm,d,\mu)$.
Assume that  $h:\bdhat\Om \to \eR$ vanishes
$\bCp$-q.e.\ on $\bdry\Om$,
i.e.\ 
\[
\bCp(\{x \in \bdy \Om : h(x) \ne 0\};\Om)=0.
\]
When $X$ is \p-hyperbolic,  assume also that  
$h(\binfty)=0$ and  that 
\begin{equation}   \label{eq-lim-f(infty)}
f(\binfty)=\lim_{\clOm \ni y\to\binfty} f(y). 
\end{equation}

Then both $f$ and $f+h$ are resolutive and $\Hp f = \Hp (f+h)$.
Moreover, if $X$ is \p-parabolic, then
the requirement~\eqref{eq-def-Perron} in the 
definition of $\lP f$ and $\uP f$ only needs to be satisfied  at
finite boundary points $x\in\bdry\Om$.
\end{thm}

\begin{proof}
Corollary~\ref{cor-Dp-equiv} 
implies that  $f\in\Dp(\clOm,\dhat,\muhtp)$.
As we shall now show, it then quite easily follows
that also $f\in\Dp(\clOm \cup \{\binfty\},\dhat,\muhtp)$.
Indeed, if $\gh\in L^p(\clOm,\muhtp)$ is an upper gradient of $f$ 
in $\clOm$ with respect to $\dhat$, 
then it is (with $\gh(\binfty)$ arbitrary)  a \p-weak upper gradient of $f$
in $\clOm \cup \{\binfty\}$ (with  respect to $\dhat$).
When $X$ is \p-parabolic, this follows from the fact that
\p-almost every nonconstant rectifiable curve in $\Xhat$ avoids $\binfty$
(because $\Cphat(\{\binfty\})=0$ by Lemma~\ref{lem-Cp=0-infty}),
see e.g.\ \cite[Lemma~3.6]{Sh-rev} or \cite[Proposition~1.48]{BBbook}.
On the other hand, if $X$ is \p-hyperbolic then 
\eqref{eq-lim-f(infty)}
gives that
\[
 |f(\ga(0))-f(\ga(1))| = \lim_{t\to0} |f(\ga(t))-f(\ga(1))|
 \le \int_\ga \gh\,d\sh
\]
for every  nonconstant 
$\dhat$-rectifiable
curve $\ga:[0,1]\to \Xhat$ with $\ga(0)=\binfty$.
 
Since $f\in\Dp(\clOm \cup \{\binfty\},\dhat,\muhtp)$, 
Theorem~\ref{thm-quasicont} implies that
$f$ is quasicontinuous.
The resolutivity and invariance of $Pf$ with respect to $(\dhat,\muhtp)$
(and hence also with respect to $(d,\mu)$)
then follow from Theorem~7.6 in Hansevi~\cite{Hansevi2}.

 Finally, if $X$ is \p-parabolic (and hence $\Cphat(\{\binfty\})=0$), 
then  applying the first part to $h=\pm \infty\chi_{\{\binfty\}}$,
 together with a simple comparison, shows that $Pf$ is independent of $f(\binfty)$.
\end{proof}

The following result improves upon Proposition~11.6 in
Bj\"orn--Bj\"orn~\cite{BBglobal}, under our assumptions.
For \p-parabolic~$X$ this
was proved in \cite[Corollary~7.9]{Hansevi2}.

\begin{cor} \label{cor-uniq-C}
Assume that \eqref{eq-Om-cond}  holds.
Let $f \in C(\bdyhat \Om)$.
If $u$ is a bounded
\p-harmonic function in $\Om$ such that
\begin{equation}   \label{eq-lim-qe}
     \lim_{\Om \ni y \to x} u(y)=f(x)
     \quad \text{for $\bCp$-q.e.\ } x \in \bdry \Om.
\end{equation}
If $X$ is \p-hyperbolic,  assume in addition that 
$ \lim_{\Om \ni y \to \binfty} u(y)=f(\binfty)$.

Then $u=Pf$.
\end{cor}

That $u=Pf$ satisfies \eqref{eq-lim-qe}
is guaranteed by the continuity of $f$ and the  Kellogg property (saying that
q.e.\ $x\in\bdy\Om$, as well as $\binfty$ in \p-hyperbolic spaces, is regular,
i.e.\ satisfies \eqref{eq-def-regular-pt} below),
see Theorem~7.1 in Bj\"orn--Hansevi~\cite{BH1}
and Lemma~11.1 in~\cite{BBglobal}.
On the other hand, in the following corollary,
a function $u$ satisfying~\eqref{eq-lim-qe} need not exist.

\begin{cor} \label{cor-uniq-Dp}
Assume that  \eqref{eq-Om-cond} holds.
Let  $f:\clOm \cup \{\binfty\} \to \eR$
be such that $f\in\Dp(\clOm,d,\mu)$.
Assume that $u$ is a bounded
\p-harmonic function in $\Om$ such that
\eqref{eq-lim-qe} holds.
If $X$ is \p-hyperbolic, assume in addition that 
\[
 \lim_{\Om \ni y \to \binfty} u(y)=\lim_{\clOm \ni y\to\binfty}f(y)=f(\binfty).
\] 
Then $u=Pf$.
\end{cor}

\begin{proof}[Proof of Corollary~\ref{cor-uniq-C}]
This follows directly from Theorem~\ref{thm-resol-C}
by taking $h=\pm\infty \chi_E$, 
where $E=\{x\in\bdyhat \Om: \text{the equality in }\eqref{eq-lim-qe} 
\text{ fails}\}$,
together with a simple comparison for Perron solutions. Namely,
\[
\uP f = \uP (f-\infty\chi_E) \le u \le \lP (f+\infty\chi_E) = \lP f \le \uP f.\qedhere
\]
\end{proof}

\begin{proof}[Proof of Corollary~\ref{cor-uniq-Dp}]
The proof is the same as for Corollary~\ref{cor-uniq-C},
but using Theorem~\ref{thm-Dp-res} instead of Theorem~\ref{thm-resol-C}.
\end{proof}

It is also possible to directly transform 
resolutivity and invariance results from Bj\"orn--Bj\"orn--Sj\"odin~\cite{BBSjodin}  
about Perron and Wiener solutions with respect to arbitrary compactifications
(such as Martin or Stone--\v{C}ech compactifications)
of bounded open sets
into similar
results for unbounded $\Om$.
For example, \cite[Proposition~7.1]{BBSjodin} shows that the so-called Sobolev--Perron solutions 
are invariant under perturbations on sets of zero $\bCp$-capacity.

\begin{remark}   \label{rmk-qmin}
Since the \p-energy is preserved under sphericalization, 
also quasiminimizers 
are preserved in just the same way as \p-harmonic functions.
(See e.g.\ Appendix~C in~\cite{BBbook} for the definition of quasiminimizers.)
This makes it possible to transfer also results about quasiminimizers
from bounded to unbounded $\Om$.
We leave it to the interested reader to pursue this path.
\end{remark}

We conclude the paper by using sphericalization to prove
Theorems~\ref{thm-trich-intro} and~\ref{thm-barrier+local-intro}.
Recall that a
point $x\in \bdhat\Om$ is \emph{regular} (with respect to $\Om$) if 
\begin{equation}    \label{eq-def-regular-pt}
\lim_{\Om \ni y \to x} Pf(y)=f(x)
\quad \text{for all }f\in C(\bdhat\Om).
\end{equation}

A superharmonic function $u$ in $\Om$ is a \emph{barrier} 
at $x_0\in\bdhat\Om$ if
\[
\lim_{\Om\ni y\to x_0} u(y)=0 \qquad \text{and} \qquad 
\liminf_{\Om\ni y\to x} u(y)>0 \quad \text{for every } x\in\bdhat\Om\setm\{x_0\}.
\]

\begin{proof}[Proof of Theorem~\ref{thm-trich-intro}]
By sphericalization, we can see $\Om$ as a bounded open
set in $(\Xhat,\dhat,\muhtp)$
and then the trichotomy
follows from the corresponding
result for boundary points in bounded open sets 
in Bj\"orn~\cite[Theorems~2.1 and~3.3(c)]{ABclass} 
(or \cite[Theorems~13.2 and~13.13(c)]{BBbook}).
\end{proof}

\begin{proof}[Proof of Theorem~\ref{thm-barrier+local-intro}]
This follows directly from the sphericalization process, together with
similar results for bounded $\Om$, provided by
Theorems~4.2 and~6.1 in Bj\"orn--Bj\"orn~\cite{BB}
(or \cite[Theorem~11.11]{BBbook}).
\end{proof}


\end{document}